\documentclass[12pt]{amsart}

\usepackage[dvipdfmx]{hyperref}

\usepackage{amscd} 
\usepackage{amsmath} 
\usepackage{amssymb} 
\usepackage{amsthm}
\usepackage{bigdelim}
\usepackage{color} 
\usepackage{enumerate}
\usepackage{graphicx}
\usepackage{mathrsfs}
\usepackage{multirow}
\usepackage[all]{xy} 
\usepackage{hyperref}
\newtheorem{thm}{Theorem}[section] 
\newtheorem{theo}[thm]{Theorem}
\newtheorem{cor}[thm]{Corollary}
\newtheorem{prop}[thm]{Proposition}

\newtheorem{lem}[thm]{Lemma}
\theoremstyle{definition} 
\newtheorem{defn}[thm]{Definition}

\newtheorem{ex}[thm]{Example} 
\theoremstyle{remark}
\newtheorem{rem}[thm]{Remark}

\newtheorem{ques}[thm]{Question}

\newtheorem{notation}[thm]{Notation}

\newtheorem{claim}{Claim}
\newtheorem{step}{Step}

\newcommand{\codim}[0]{\operatorname{codim}}

\newcommand{\rank}[0]{\operatorname{rank}}

\newcommand{\pr}{{\rm pr}}
\newcommand{\id}{{\rm id}}

\newcommand{\I}[1]{\mathcal{I}(#1)}
\newcommand{\Aut}[1]{\mathrm{Aut}(#1)}
\newcommand{\Ker}[1]{\mathrm{Ker}(#1)}
\newcommand{\Image}[1]{\mathrm{Im}(#1)}

\title[On asymptotic base loci of relative anti-canonical divisors]
{On asymptotic base loci \\ 
of relative anti-canonical divisors \\ 
of algebraic fiber spaces}

\author{Sho EJIRI}
\address{Department of Mathematics, Graduate School of Science, Osaka University, Toyonaka, Osaka 560-0043, Japan.}
\email{{\tt shoejiri.math@gmail.com, s-ejiri@cr.math.sci.osaka-u.ac.jp}}

\author{Masataka IWAI}
\address{Department of Mathematics, Graduate School of Science, Osaka City University 3-3-138, Sugimoto, Sumiyoshi-ku Osaka, 558-8585
Japan} 
\email{{\tt masataka.math@gmail.com, masataka@ms.u-tokyo.ac.jp, masataka@sci.osaka-cu.ac.jp}}

\author{Shin-ichi MATSUMURA}
\address{Mathematical Institute, Tohoku University, 
6-3, Aramaki Aza-Aoba, Aoba-ku, Sendai 980-8578, Japan.}
\email{{\tt mshinichi-math@tohoku.ac.jp, mshinichi0@gmail.com}}

\date{\today, version 0.01}

\subjclass[2010]{Primary 14D06, Secondary 14E30, 32J25}

\keywords
{Klt pairs, 
Lc pairs, 
Anti-canonical divisors, 
Relative anti-canonical divisors, 
Augmented base loci, 
Restricted base loci, 
Stable base loci, 
Upper level sets of Lelong numbers, 
Rational curves, 
MRC fibrations, 
Direct image sheaves, 
Singular hermitian metrics, 
Numerical dimension, 
Numerical flatness, 
Hermitian flatness
\'etale trivialization.}

\begin{document}
\maketitle
\begin{abstract}
In this paper, we study the relative anti-canonical divisor $-K_{X/Y}$ 
of an algebraic fiber space $\phi: X\to Y$, and we reveal relations among 
positivity conditions of $-K_{X/Y}$,  
certain flatness of  direct image sheaves, 
and variants of  the base loci  
including the stable (augmented, restricted) base loci and  upper level sets of Lelong numbers. 
This paper contains three main results: 
The first result says that 
all the above base loci are located in the horizontal direction unless they are empty.  
The second result is an algebraic proof 
for Campana--Cao--Matsumura's equality 
on Hacon--M\textsuperscript cKernan's question, 
whose original proof depends on analytics methods. 
The third result {partially solves the question which asks whether} 
algebraic fiber spaces with semi-ample relative anti-canonical divisor 
actually have a product structure 
via the base change by an appropriate finite \'etale cover of $Y$. 
Our  proof is based on algebraic as well as analytic methods 
for positivity  of direct image sheaves.

\end{abstract}
\tableofcontents
\section{Introduction}\label{section:intro}
\subsection{Relative anti-canonical divisors}

This paper studies an algebraic fiber space $\phi:X\to Y$ 
between projective varieties over the complex number field
(that is, a surjective morphism 
with connected fibers) and its relative anti-canonical divisor 
$-K_{X/Y}:=-K_X +\phi^*K_Y$. 
The total space $X$ and the base $Y$ are assumed to be smooth in this section, 
but the case where they are singular is also treated in this paper. 
The geometric structure of $\phi: X \to Y$ is known to be 
deeply connected with several positivity conditions 
of $-K_{X/Y}$. 
Positivity conditions of $-K_{X/Y}$ in algebraic geometry 
can be measured by the {\textit{asymptotic base loci}} of $-K_{X/Y}$. 
Here the asymptotic base loci mean the stable base locus 
$\mathbb B(-K_{X/Y})$ and its approximative variants introduced by \cite{ELMNP1}:  
the augmented base locus $\mathbb B_+(-K_{X/Y})$ and 
the restricted base locus $\mathbb B_-(-K_{X/Y})$. 
The semi-ampleness (resp. ampleness, nefness) is 
equivalent to the condition $\mathbb B(\bullet) =\emptyset$ 
(resp. $\mathbb B_+(\bullet)=\emptyset$, $\mathbb B_-(\bullet)=\emptyset$).

Our interest is to understand how restricted 
the geometric structure of $\phi: X \to Y$ is 
when $-K_{X/Y}$ satisfies certain positivity, 
in other words, to reveal relations 
between the asymptotic base loci and the geometric structure of $\phi: X \to Y$. 
Regarding this, there are several known results. 
When $-K_{X/Y}$ is nef, 
an argument due to Cao--H\"oring (\cite{CH19}) 
determines the detailed geometric structure of $\phi: X \to Y$ 
including the local triviality (cf. \cite[\S A]{PZ19}).  
%
The celebrated work due to Koll\'ar--Miyaoka--Mori 
(\cite[Corollary~2.8]{KoMM92}) tells us that, 
in the case when $Y$ is \textit{not} one point, 
$-K_{X/Y}$ is \textit{not} ample (i.e., $\mathbb B_+(-K_{X/Y})\ne\emptyset$). 
In the same case, Deng's result (\cite[Theorem E]{Den17}) further shows that $\mathbb B_+(-K_{X/Y})$ is dominant over $Y$. 
All the above mentioned results 
can be generalized to klt pairs. 
As a natural question, 
it arises the problem of studying 
the case of lc pairs or the case when $\mathbb B_-(-K_{X/Y})\ne\emptyset$. 

The aim of this paper is to systematically understand 
such problems  and deeper relations 
between the geometric structure of $\phi$ and 
positivity conditions on $-K_{X/Y}$. 
To this end, we focus on the various base loci of $-K_{X/Y}$ 
and apply the theory of positivity of direct image sheaves. 
\subsection{On augmented base loci}
\label{intro2}

This subsection introduces the results related to 
the augmented base locus $\mathbb B_+(-K_{X/Y})$. 
The following theorem was firstly proved by Deng (\cite[Theorem E]{Den17}). 
%
\begin{theo}[A special case of \textup{Corollary~\ref{cor:dominate}}] \label{thm:1_intro}
If $\mathbb B_+(-K_{X/Y})$ is not dominant over $Y$, then $Y$ is a point. 
In particular, if $\dim Y>0$, then $\dim \mathbb B_+(-K_{X/Y}) \ge \dim Y$. 
\end{theo}
Note that Deng's proof uses an analytic method, but we give an algebraic proof of this theorem. 
Furthermore, the full statement of Theorem~\ref{thm:1_intro} includes 
the case where $X$ has at worst lc singularities (see Corollary~\ref{cor:dominate}), which does not follow from \cite[Theorem E]{Den17}. 
The behavior of $\mathbb B_+(-K_{X/Y}-\Delta)$ 
differs according to the singularity of the pair $(X,\Delta)$
(see Example \ref{ex:nef big but not ample}). 
%
Hence Theorem \ref{thm:1.5_intro} seems to be an extremely generalized result. 

Theorem~\ref{thm:1_intro} is one of corollaries to the theorem below. 

\begin{theo}[\textup{Theorem~\ref{thm:inclusions}}] \label{thm:1.5_intro}
Fix a smooth fiber $F$ of smallest dimension. 
Let $D$ be a divisor on $X$ such that $\mathcal O_F(D)$ is nef $($resp. ample$)$. 
If $\mathbb B_-(D)$ $($resp. $\mathbb B_+(D)$$)$ intersects $F$,  
then so does $\mathbb B_-(rD-K_{X/Y})$ $($resp. $\mathbb B_+(rD-K_{X/Y})$$)$ for any positive rational number $r$. 
\end{theo}
{Using this theorem with $D=K_{X/Y}$ or $K_X$, one obtains the following corollary:}
\begin{cor}[\textup{\cite[Corollary~1.3]{Pat14}}] \label{cor:1_intro}
{ 
Fix a smooth fiber $F$ of smallest dimension. 
Assume that $K_F$ is nef. 
Then $F$ and $\mathbb B_-(K_{X/Y})$ do not intersect. 
Suppose further that $K_Y$ is nef. 
Then any proper curve $C \subset X$ with $K_X\cdot C <0$ does not intersect~$F$. 
}
\end{cor}

\subsection{On restricted base loci and stable base loci}
\label{intro3}

It is natural to consider 
the restricted base locus $\mathbb B_-(-K_{X/Y})$ and other base loci, 
as the next problem of Cao--H\"oring's critical work for the case $\mathbb B_-(-K_{X/Y})=\emptyset$. 
The following result can be seen as an analogue of 
Theorem \ref{thm:1_intro} to various base loci 
(restricted base loci, stable base loci, and 
upper level sets of Lelong numbers). 

\begin{theo}[\textup{Theorem~\ref{thm:nefness}, 
Theorem~\ref{thm:nonLelongnumber}, 
Theorem~\ref{thm:semiampleness}}]\label{thm:2_intro}
Let $X$ and $Y$ be projective manifolds  
and $\phi: X \to Y$ be a surjective morphism with connected fibers. 
Then we have:
\begin{itemize}
\item[$\bullet$] If $\mathbb B_-(-K_{X/Y})$ is not dominant over $Y$, 
then $\mathbb B_-(-K_{X/Y})$ is empty $($that is, $-K_{X/Y}$ is nef$)$. 

\item[$\bullet$] If $\mathbb B(-K_{X/Y})$ is not dominant over $Y$, 
then $\mathbb B(-K_{X/Y})$ is empty $($that is, $-K_{X/Y}$ is semi-ample$)$. 

\item[$\bullet$] Let $h$ be a singular hermitian metric on $-K_{X/Y}$ with semipositive curvature and let $P(h)$ denote the set of points at which $h$ has positive Lelong number. 
If $P(h)$ is not dominant over $Y$, then $P(h)$ is empty 
$($that is, the Lelong number of $\sqrt{-1}\Theta_h$ is zero everywhere$)$. 

\end{itemize}
\end{theo}

The direct image $\phi_{*}(\tilde{A})$ of an appropriate relatively ample divisor 
$\tilde{A}$ on $X$ satisfies numerical flatness if $-K_{X/Y}$ is nef 
by \cite{CH19, CCM}. 
For the proof of Theorem~\ref{thm:2_intro}, 
we confirm that the same conclusions hold  
under the slightly weaker assumption that 
the restricted base locus $\mathbb B_-(-K_{X/Y})$ is not dominant over $Y$ 
(see Theorem~\ref{r-thm:locallytrivial}). 
In this process, we can find a relation 
between positivity conditions of $-K_{X/Y}$ and 
certain flatness of the direct image sheaves, 
e.g. numerical flatness, hermitian flatness, and \'etale trivializability. 
Further we clarify that the certain flatness recovers 
the nefness, semi-ampleness, or the property of $P(h)=\emptyset$.

\subsection{On invariants of nef relative anti-canonical divisors} \label{intro5}
Motivated by Hacon--M\textsuperscript cKernan's question (\cite{HM07}), 
the Iitaka--Kodaira dimension $\kappa(-K_{X/Y})$ and the numerical Kodaira dimension $\mathrm{nd}(-K_{X/Y})$ of \textit{nef} relative anti-canonical divisors $-K_{X/Y}$ have been studied, 
and we now have the following relations:
\begin{align*}
\xymatrix@R=10pt@C=10pt{ 
\kappa(-K_{X/Y}) \ar@{}[r]|-{\le} 
\ar@{}[d]|-{\textup{\rotatebox[origin=c]{90}{$\ge$}}}_-{(1)~}
& \mathrm{nd}(-K_{X/Y})  \ar@{}[r]|-{\le} \ar@{=}[d]_-{(2)~}
& \dim X \ar@{}[d]|-{\textup{\rotatebox[origin=c]{90}{$\le$}}} 
\\  \kappa(-K_F) \ar@{}[r]|-{\le} 
& \mathrm{nd}(-K_F) \ar@{}[r]|-{\le} 
& \dim F
}
\end{align*}
The inequality~(1) was showed by Ejiri--Gongyo (\cite{EG19}). 
The equality~(2) was proved by 
Campana--Cao--Matsumura (\cite{CCM}), 
whose methods heavily depends on analytic methods 
based on metric positivity of direct image sheaves 
(\cite{BP08, HPS18, PT18}, and references therein). 
%

This paper gives an algebraic proof of 
Campana--Cao--Matsumura's equality~(2) and 
slightly generalize it to the case where $Y$ has only canonical singularities. 

\begin{thm}[Theorem~\ref{thm:nd}]\label{thm:nd_intro}
Let $X$ be a normal projective variety and let $\Delta$ be an effective $\mathbb Q$-divisor on $X$ such that $(X,\Delta)$ is klt. 
Let $\phi:X\to Y$ be a morphism with connected fibers onto
a $\mathbb Q$-Gorenstein normal projective variety $Y$ 
with only canonical singularities. 
If $-(K_{X/Y}+\Delta)$ is nef, then we have 
\begin{align*}
\mathrm{nd}(-(K_{X/Y}+\Delta))=\mathrm{nd}(-(K_{F}+\Delta|_F)) 
\end{align*}
for a general fiber $F$. 
\end{thm}

\subsection{On the structure of semi-ample relative anti-canonical bundles}
\label{intro4}

The situation where $-K_{X/Y}$ is semi-ample can be expected to occur 
only in a more restricted case 
than the case of $\phi$ being locally trivial. 
Here we propose the following question:
\begin{ques} \label{ques:semi-ample}
If $-K_{X/Y}$ is semi-ample, then does 
there exist a finite \'etale cover $Y'\to Y$
such that the morphism $F\times_{\mathbb C}Y' \to Y$ induced by the second projection factors through $\phi:X\to Y$?
$$
\xymatrix{{F \times_{\mathbb C} Y'} \ar@{->}[r] \ar[d]_-{\mathrm{pr}_1} \ar@{}[dr]|\circlearrowleft & X \ar[d]^-\phi \\
	Y' \ar[r] & Y. 
}
$$
Here $F$ is a fiber of $\phi$. 
\end{ques}
For example, it follows that 
the above question is affirmative when $-K_{X/Y}$ is trivial and 
the irregularity of $F$ is zero 
from \cite[Theorem 5.8]{LPT18} and \cite[Lemma 6.4]{Dru17}. 
Note that Example \ref{ex:semi-ample} tells us 
that $\phi: X \to Y$ itself is not necessarily a product 
without taking an \'etale cover $Y' \to Y$. 
We give a partial affirmative answer to Question~\ref{ques:semi-ample}: 
\begin{theo} \label{thm:4_intro}
Suppose that $-K_{X/Y}$ is semi-ample. 
Then the direct image 
$$
\phi_*(-mK_{X/Y})
$$ 
is \'etale trivializable for every $m \in \mathbb{Z}_{>0}$. 
More precisely, there exists a finite \'etale cover $Y' \to Y$ such that 
the pull-back of $\phi_*(-mK_{X/Y})$ to $Y'$ is a trivial vector bundle. 
Moreover, the following hold:
\begin{itemize}
\item[(1)] If the anti-canonical bundle of $-K_F$ on a fiber $F$ is ample, 
then there exists a finite \'etale cover $Y' \to Y$ such that the fiber product $X':=Y' \times_Y X$ is isomorphic to the product of $Y' \times_{\mathbb C} F$ as a $Y'$-scheme. 

\item[(2)] If the irregularity of a general fiber of 
the Iitaka 
fibration associated to $-K_F$ is equal to zero, 
then the same conclusion as in $(1)$ holds. 
\end{itemize}
\end{theo}

Our proof is based on a geometric comparison among the original morphism $\phi: X \to Y $ 
and the Iitaka fibrations of $-K_{X/Y}$ and $-K_F$, 
which reduces the general case to the extremal case where $-K_F$ is ample 
or $K_{X/Y}$ is trivial via the direct image $\phi_*(-mK_{X/Y})$. 

\begin{cor}\label{cor-main}
Let $X:=\mathbb{P}(E) \to Y$ be the projective space bundle over a projective manifold $Y$. 
If $-K_{X/Y}$ is semi-ample, then there exists a finite \'etale cover $Y'\to Y$ such that the fiber product of $X$ and $Y'$ over $Y$ is isomorphic to the product of the projective space and $Y'$. 
\end{cor}

This paper is organized as follows: 
In Section \ref{section:preliminaries}, 
we recall the definition and properties of the asymptotic base loci. 
Section \ref{section:base loci of -K}, \ref{section:small B-}, and \ref{subsection:semi-ample -K}
are respectively devoted to topics explained in
Section \ref{intro2}, \ref{intro3}, and \ref{intro4}. 
In Section \ref{section:ex}, 
we collect examples, which help us to understand our results. 
In Section \ref{section:appendix}, 
we give an analytic proof for some results 
in Section \ref{section:base loci of -K} and \ref{section:small B-}. 

\smallskip
 
After we have finished to write this paper, 
Yoshinori Gongyo informed the authors that 
Theorem \ref{thm:4_intro} follows from \cite[Proposition 4.4, Theorem 4.7]{Amb05}. 
However our argument is quite different from that of Ambro and 
gives more geometric  proof,  
and thus we believe that it is worth to displaying it in this paper.

\subsection*{Acknowledgements}
The authors wish to express their thanks to Prof. Kento Fujita 
for constructing an interesting example (Example~\ref{ex:Fujita's example}) 
and pointing out an error in the proof of Lemma~\ref{lem:pullback of base loci}, 
and also to Prof. Yoshinori Gongyo for informing them 
that Theorem \ref{thm:4_intro} follows from \cite[Proposition 4.4, Theorem 4.7]{Amb05}. 
S. E. is grateful to Prof. Osamu Fujino 
for helpful comments on the content of Section~\ref{section:base loci of -K}. 
He was supported by JSPS KAKENHI Grant $\sharp$18J00171. 
M. I was supported by JSPS KAKENHI Grant $\sharp$17J04457, 
by the Program for Leading Graduate Schools, MEXT, Japan
and by the public interest incorporated foundation Fujukai.
S. M. would like to express his thanks to 
Prof. Masaaki Murakami, Prof. Gregory Kumar Sankaran, and Prof. Taro Sano 
for helpful comments on the moduli of abelian varieties, 
and he would like to thank Prof. Takato Uehara and Prof. Takayuki Koike 
for discussion on Example \ref{kodaira}. 
He was supported by the Grant-in-Aid 
for Young Scientists (A) $\sharp$17H04821 and 
Fostering Joint International Research (A) $\sharp$19KK0342 from JSPS. 

\section{Preliminaries} \label{section:preliminaries}

\subsection{Notations and conventions} \label{subsection:notaion}
Throughout this paper, we work over the field $\mathbb C$ of complex numbers. 
We use the following conventions:  
A \textit{variety} is an integral separated scheme of 
finite type over $\mathbb C$, 
a \textit{curve} is a 1-dimensional variety, 
and a \textit{manifold} is a smooth variety. 
We interchangeably use the words 
\lq \lq Cartier divisors",
\lq \lq invertible sheaves", and 
\lq \lq line bundles".

\subsection{Asymptotic variants of base loci}
\label{subsection:base loci}
Let $D$ be a $\mathbb Q$-Cartier divisor on a projective variety $X$. 
We recall the definitions and several properties of the stable base locus $\mathbb B(D)$ of $D$ and its approximative variants: the augmented base locus $\mathbb B_+(D)$ and restricted base locus $\mathbb B_-(D)$. 
\begin{defn}
{Let $D$ be a $\mathbb Q$-Cartier divisor on a projective variety $X$ and take an integer $i>0$ so that $iD$ is Cartier. }
Let $\mathrm{sp}(X)$ denote the underlying Zariski-topological subset of $X$. 
(1) The \textit{stable base locus} $\mathbb B(D)$ is defined by 
$$
\mathbb B(D):=\bigcap_{m\ge1}\mathrm{Bs}(imD)_{\mathrm{red}} \subseteq \mathrm{sp}(X).
$$
(2) The \textit{restricted base locus} (or \textit{diminished base locus}) 
$\mathbb B_-(D)$ and the \textit{augmented base locus} $\mathbb B_+(D)$ are defined by
$$
\mathbb B_-(D):=\bigcup_{A}\mathbb B(D+A) \subseteq \mathrm{sp}(X) 
\text{ \quad and \quad}
\mathbb B_+(D):=\bigcap_{A}\mathbb B(D-A) \subseteq \mathrm{sp}(X), 
$$ 
where $A$ runs over all the $\mathbb Q$-Cartier ample divisors on $X$. 
\\
(3) Let $X$ be a projective manifold and $(L,h)$ be a line bundle 
with a singular hermitian metric with semipositive curvature current.
The \textit{Lelong number} $\nu (h,x)$ of $h$ at $x \in X$ is defined by
$$
\nu (h,x) := \liminf_{z \rightarrow x}\frac{\varphi(z)}{\log|z-x|},
$$
where $\varphi$ is a local weight of $h$.
The \textit{upper level set of positive Lelong numbers} $P(h)$ is defined by
$$P(h) : = \{ x \in X  \,|\, \nu(h,x) >0 \}.$$
\end{defn}
In general we have the following inclusions 
$$
\mathbb B_-(L) \subseteq P(h)
\text{ \quad and \quad}
\mathbb B_-(L) \subseteq  \mathbb B(L) \subseteq  \mathbb B_+(L)  
$$
for any singular hermitian line bundle $(L, h)$ on a projective manifold. 
We remark that $P(h)$ and $\mathbb{B}_{-}(L)$ is 
a countable union of Zariski-closed sets. 

The following preliminary lemma 
compares the restricted (and augmented) base locus of a divisor 
with that of its pullback via a surjective morphism. 

%
\begin{lem} \label{lem:pullback of base loci}
Let $\phi :X\to Y$ be a surjective morphism from a projective variety $X$ to a normal projective variety $Y$. 
Let $D$ be a $\mathbb Q$-Cartier divisor on $Y$ and 
$V$ be a Zariski-open subset of $Y$ such that 
$U:=\phi^{-1}(V) \to V$ is flat over $V$. 
Then the following holds:
\begin{itemize}
\item[\rm (1)] $ V\cap\mathbb B_-(D) = V\cap \phi \left(\mathbb B_-( \phi^*D)\right)$.
\item[\rm (2)] $ V\cap\mathbb B_+(D) \subseteq V\cap \phi \left(\mathbb B_+(\phi^{*}D)\right)$.
\end{itemize}
\end{lem}
\begin{proof}
We prove (1). 
We first consider the inclusion ``$\supseteq$''. 
Let $A$ be an ample $\mathbb Q$-Cartier divisor on $X$. 
Take an ample $\mathbb Q$-Cartier divisor $H$ on $Y$ so that $A-\phi^*H$ is ample.  
Then we have 
\begin{align*}
\mathbb B(\phi^*D +A)
& = \mathbb B(\phi^*D +\phi^*H +A-\phi^*H)
\\ & \subseteq \mathbb B(\phi^*(D +H)) 
\subseteq \phi^{-1}(\mathbb B(D+H)) 
\subseteq \phi^{-1}(\mathbb B_-(D)), 
\end{align*}
and thus $ \mathbb B_-(\phi^*D)=\bigcup_A \mathbb B(\phi^*D +A) 
\subseteq \phi^{-1}(\mathbb B_-(D)).$ 
The assertion can be obtained from the surjectivity of $\phi$.
We can prove (2) by a similar argument, and thus we omit it.  
 
Next we show the converse inclusion ``$\subseteq$''. 
For a fixed $v \in V \setminus \phi \left(\mathbb B_-(\phi^*D)\right)$, 
we check that $v \notin \mathbb B_-(D)$. 
Take a closed subvariety $X'\subseteq X$ so that $\dim X'=\dim Y$ and 
that the induced morphism $\phi':X'\to Y$ is flat over $v$. 
We have 
$ \mathbb B_-({\phi'}^*D) \subseteq X' \cap \mathbb B_-(\phi^*D) $ 
by the above argument. 
Replacing $X$ with $X'$, we may assume that $\phi$ is generically finite. 
Let $A$ be an ample Cartier divisor on $X$ 
and $H$ an ample Cartier divisor on $Y$ such that 
$\phi^{-1}(v) \cap \mathrm{Supp}(\phi^*H-A) = \emptyset$.
Such an $H$ exists, as $\phi$ is finite over a neighborhood of $v$. 
Take $m\in\mathbb Z_{>0}$. 
Then 
$$
X\setminus \phi^{-1}(v) 
\supseteq \mathbb B_-(\phi^*D)
\supseteq \mathbb B(m\phi^*D +A)
=\mathrm{Bs}(lm\phi^*D +lA)
$$
for some $l>0$, so $\phi^{-1}(v) \cap \mathrm{Bs}(l\phi^*(mD +H))$,
and thus there is a morphism 
$\alpha: \bigoplus^{\deg \phi} \mathcal O_X \to \mathcal O_X(l\phi^*(mD +H))$
that is surjective over $\phi^{-1}(v)$. 
Since $\phi$ is affine over a neighborhood of $v$, 
the composite $\beta$ of the morphisms
$$
\bigoplus^{\deg \phi} \phi_*\mathcal O_X 
\xrightarrow{\phi_*(\alpha)}
\phi_*\mathcal O_X(l\phi^*(mD +H))
\cong \left(\phi_*\mathcal O_X\right) \otimes \mathcal O_Y(l(mD +H))
$$
is surjective over $v$. 
Furthermore, since $\phi_*\mathcal O_X$ is free at $v$, the natural morphism 
$$
\gamma: (\phi_*\mathcal O_X )^\vee \otimes \phi_*\mathcal O_X \to \mathcal O_Y
$$
is surjective over $v$.  
Take $n\in\mathbb Z_{>0}$ so that 
$
\left( (\phi_*\mathcal O_X)^\vee \otimes \phi_*\mathcal O_X \right)(nH)
$
is globally generated. 
Then we have the following morphisms:
\begin{align*}
\bigoplus^{\deg \phi}
\left( (\phi_*\mathcal O_X)^{\vee} 
\otimes 
\phi_*\mathcal O_X \right)(nH)
\cong &
(\phi_*\mathcal O_X)^{\vee} 
\otimes 
\left( \bigoplus^{\deg \phi} \phi_*\mathcal O_X \right)(nH)
\\ \xrightarrow{\textup{by $\beta$}}  
& \left( (\phi_*\mathcal O_X)^{\vee} \otimes \phi_*\mathcal O_X \right)
(lmD +(l+n)H) 
%
\\ \xrightarrow{\textup{by $\gamma$}}  
& \mathcal O_Y(lmD +(l+n)H). 
\end{align*}
These morphisms are surjective over $v$, so 
$v \notin \mathbb B(lmD+(l+n)H) =\mathbb B\left(D +\frac{l+n}{lm}H\right)$. 
Then it follows that $v\notin\mathbb B_-(D)$ 
since $m$ can be chosen as large as we want. 
\end{proof}

\section{Augmented and restricted base loci for lc pairs} 
\label{section:base loci of -K}

\subsection{On augmented and restricted base loci for lc pairs}
The aim of this subsection is to 
prove Theorem \ref{thm:inclusions} and 
its corollaries (including Theorem \ref{thm:1_intro} and \ref{thm:1.5_intro}). 
We emphasize that the argument in this subsection works 
not only for non-singular cases (or klt pairs) but also for lc pairs. 
We use the following notation:
\begin{notation} \label{notation:mor}
Let $\phi:X\to Y$ be a surjective morphism with connected fibers 
from  a normal projective variety $X$ 
to a normal $\mathbb Q$-Gorenstein (that is, $K_Y$ is $\mathbb Q$-Cartier) projective variety $Y$ with only canonical singularities. 
Let $\Delta$ be an effective $\mathbb Q$-divisor on $X$ 
such that $K_X+\Delta$ is $\mathbb Q$-Cartier. 
We denote by $V(\phi,\Delta)$ the subset of $Y$ consisting of \textit{regular} points $v$ with the following properties:  
\\ \noindent \quad {(1) 
The pair $(X,\Delta)$ is log canonical at every point in $\phi^{-1}(v)$. }
\\ \noindent \quad (2) 
There exists a birational morphism $\pi:X'\to X$ from a normal projective variety $X'$ such that $\phi':=\phi\circ\pi:X'\to Y$ is smooth over a neighborhood of $v$.  
\\ \noindent \quad (3) 
There exists an effective $\pi$-exceptional \textit{reduced} Weil divisor $E$ on $X'$ 
such that 
\begin{itemize}
\item 
$
\pi^{-1}_*\Delta +E 
\ge \pi^*(K_X +\Delta) -K_{X'}, 
$ 
and 
\item {${\phi'}^{-1}(v)$ intersects transversally with each component of $\pi^{-1}_*\Delta +E$. }
\end{itemize}
\end{notation}
\begin{rem} \label{remark:mor}
\noindent (1) The set $V(\phi,\Delta)$ is a Zariski-open set in $Y$. 
\\ \noindent (2) The pair $(X,\Delta)$ has 
at worst log canonical singularities at every point in $\phi^{-1}(V(\phi,\Delta))$. 
Furthermore, if the non-log canonical locus of $(X,\Delta)$ is not dominant over $Y$, then $V(\phi,\Delta)$ is automatically non-empty. 
\\ \noindent(3)  
The morphism $\phi: X \to Y$ is flat over $V(\phi,\Delta)$. 
\end{rem}
The following is one of our main results:
\begin{thm} \label{thm:inclusions}
Let the notation be as in Notation \ref{notation:mor}. 
Fix a point $v$ in $V(\phi,\Delta)$. 
Let $D$ be a $\mathbb Q$-Cartier divisor on $X$ and let 
$i \in \mathbb{Z}_{>0}$ be an integer such that $iD$ is Cartier. 
Then we have: 
\begin{itemize}
\item[\rm(1)] Suppose that  $\mathcal O_{X_v}(iD)$ is nef. 
If $\mathbb B_-(D)$ intersects $\phi^{-1}(v)$, then so does $\mathbb B_-(D-(K_{X/Y}+\Delta))$. 
\item[\rm(2)] Suppose that  $\mathcal O_{X_v}(iD)$ is ample. 
If $\mathbb B_+(D)$ intersects $\phi^{-1}(v)$, then so does $\mathbb B_+(D-(K_{X/Y}+\Delta))$. 
\end{itemize}
\end{thm}
%
To prove Theorem~\ref{thm:inclusions}, we need Proposition~\ref{prop:base locus}, which is an application of weak positivity theorems developed by~\cite{Fuj78, Kaw81, Vie83, Cam04, Fuj17}.  
We employ \cite[Theorem~E]{DM19}, since it deals with log canonical pairs 
and gives explicitly a Zariski-open subset {over} which direct images of relative {log} pluricanonical bundles are weakly positive. 
\begin{prop} \label{prop:base locus}
Let the notation be as in Notation \ref{notation:mor}. 
Let $V$ be a Zariski-open subset in $V(\phi,\Delta)$ and put $U:=\phi^{-1}(V)$. 
Suppose that $K_U+\Delta|_U$ is relatively ample over $V$. 
For an ample $\mathbb Q$-Cartier divisor $H$ on $Y$, 
we have 
$$
\mathbb B(K_{X/Y}+\Delta+\phi^*H)\cap U=\emptyset.  
$$ 
\end{prop}
\begin{proof}[Proof of Proposition \ref{prop:base locus}] \setcounter{step}{0}
Fix $v\in V$. We prove 
$\mathbb B(K_{X/Y}+\Delta+\phi^*H)\cap f^{-1}(v)=\emptyset.$ 
\begin{step} \label{step:notation}
We set up the notation used in the proof. 
Let $\tau:Y^\flat\to Y$ be a resolution of singularities 
passing through the flattening of $\phi$. 
Let $X^\flat$ be the normalization of the main component of $X\times_Y Y^\flat$. 
Let $\pi:X^\sharp \to X$ be a log resolution of $(X,\Delta)$ 
passing through $X^\flat \to X$. 
We define the notation by the following commutative diagram:
\begin{align} 
\label{diag:1}
\xymatrix{ X \ar[d]_-\phi & X^\flat \ar[l]_\rho \ar[d]_-{\phi^\flat} & X^\sharp \ar@/_20pt/[ll]_\pi \ar[l]_-\sigma \ar[dl]^-{\phi^\sharp} \\
Y & Y^\flat. \ar[l]^-\tau & 
}
\end{align}
%
%
Let $E$ and $F$ be two effective $\pi$-exceptional $\mathbb Q$-divisors on $X^\sharp$ with no common components such that 
$$
K_{X^\sharp}+\pi^{-1}_*\Delta +E = \pi^*(K_X+\Delta) +F. 
$$
It follows that $E':=K_{Y^\flat}-\tau^*K_Y\ge0$ 
since $Y$ has only canonical singularities. 
Put $\Delta^\sharp:=\pi^{-1}_*\Delta+E+{\phi^\sharp}^*E'$. 
Then we have 
$$
K_{X^\sharp/Y^\flat}+\Delta^\sharp = \pi^*(K_{X/Y}+\Delta) +F. 
$$
Let $k$ be an integer large and divisible enough. 
Set $P^\sharp:=k(K_{X^\sharp/Y^\flat}+\Delta^\sharp)$. 
Put $V^\flat:=\tau^{-1}(V)$ and $U^\flat:=\rho^{-1}(U)=(\phi^\flat)^{-1}(V^\flat)$. 
 
We may assume the following conditions.
\begin{itemize}
\item The morphism $\tau$ (resp. $\rho$) is an isomorphism over $V$ (resp. $U$). 
\item The divisor $\Delta^\sharp$ is simple normal crossing. 
\item The morphism $\phi^\sharp$ is smooth over a neighborhood of $\tau^{-1}(v)$. 
\item The divisor $\Delta^\sharp|_{X^\sharp_v}$ is simple normal crossing and its each coefficient is at most one, by the definition of $V(\phi,\Delta)$ in Notation~\ref{notation:mor};
\item In particular, each coefficient of horizontal components 
(that is, components dominating $Y$) in $\Delta^\sharp$ is at most one.
\end{itemize}
%
%
\end{step}
\begin{step} \label{step:loc free}
We prove that $\phi^\sharp_*\mathcal O_{X^\sharp}(P^\sharp)$ is locally free over $V^\flat$. 
Since 
$$
P^\sharp =k{\pi}^*(K_{X/Y}+\Delta)+kF
$$ 
and $F$ is $\pi$-exceptional, 
we have $\pi_*\mathcal O_{X^\sharp}(P^\sharp)\cong\mathcal O_X(k(K_{X/Y}+\Delta))$. 
Hence the claim of this step is equivalent to saying the local freeness of 
$$
\tau_*\phi^\sharp_*\mathcal O_{X^\sharp}(P^\sharp)
\cong \phi_*\pi_*\mathcal O_{X^\sharp}(P^\sharp)
\cong \phi_*\mathcal O_{X}(k(K_{X/Y}+\Delta))
$$
over $V$, since $\tau$ is an isomorphism over $V$. This follows from the flatness of $\phi$ over $V$ and the relative ampleness of $(K_{X/Y}+\Delta)|_U$ over $V$. 
\end{step}
%
%
\begin{step} \label{step:DM}
Shrinking $V$ if necessarily, we show that the sheaf
\begin{align*}
\left( \phi^\sharp_*\mathcal O_{X^\sharp}(P^\sharp) \right)^{[s]}
\otimes 
\mathcal O_{Y^\flat}\left(ks {\tau}^*H\right) 
\end{align*}
is globally generated on $V^\flat$ for $s\gg1$. 
Here $(\bullet )^{[s]}$ denotes 
the double dual $\left((\bullet)^{\otimes s}\right)^{\vee\vee}$ 
of the $s$-times tensor. 
Let $H^\flat$ be a very ample divisor on $Y^\flat$. 
Thanks to \cite[Theorem~E]{DM19}, we find a neighborhood of $\rho^{-1}(v)$ on which
\begin{align}
\left( \phi^\sharp_*\mathcal O_{X^\sharp}(P^\sharp) \right)^{[s]}
\otimes 
\mathcal O_{Y^\flat}\left(kK_{Y^\flat}+k(\dim Y+1)H^\flat\right)
\label{sheaf:1}
\end{align}
is globally generated for any $s\ge1$.
Shrinking $V$ again, we may assume that these sheaves are globally generated over $V^\flat$. 
If $s\gg0$, then the sheaf 
\begin{align}
\mathcal O_{Y^\flat}\big(ks \tau^*H -kK_{Y^\flat}-k(\dim Y+1)H^\flat\big) 
\label{sheaf:2}
\end{align}
is globally generated on $V^\flat$, 
since $\tau$ is an isomorphism over $V$. 
By taking the tensor product of (\ref{sheaf:1}) and (\ref{sheaf:2}), 
we obtain the claim of this step. 
\end{step}
\begin{step} \label{step:free}
Put $P^\flat:=\sigma_*P^\sharp =k\rho^*(K_{X/Y}+\Delta) +k\sigma_* F$. 
Note that $P^\flat$ is a Weil divisor that is not necessarily $\mathbb Q$-Cartier. 
We prove that $\mathcal O_{X^\flat}(sP^\flat+ks{\rho}^*\phi^*H)$ is globally generated on $U^\flat$. 
Let $W\subseteq Y^\flat$ be the largest Zariski-open subset on which $\phi^\sharp_*\mathcal O_{X^\sharp}(P^\sharp)$ is locally free. 
Set $T:=(\phi^\flat)^{-1}(W)$ and $h:=\phi^\flat|_{T}:T\to W$. 
Then, by Step~\ref{step:loc free}, we have $V^\flat \subseteq W$ and $U^\flat\subseteq T$. 
Since $\phi^\flat$ is an equi-dimensional morphism, we have 
$$
\mathrm{codim}_{X^\flat}(X^\flat\setminus T) =\mathrm{codim}_{Y^\flat}(Y^\flat\setminus W) \ge 2, 
$$ 
and thus the claim of this step is equivalent to saying that 
$\mathcal O_{T}((sP^\flat+ks\rho^*\phi^*H)|_T)$
is globally generated over $U^\flat$. 
{
Since $P^\flat\sim k\rho^*(K_{X/Y}+\Delta)+k\sigma_*F$ and $\sigma_*F$ is $\rho$-exceptional, it follows that 
$P^\flat|_{U^\flat}\sim k\rho^*(K_{X/Y}+\Delta)|_{U^\flat}$, 
and so the natural inclusion 
\begin{align}
\sigma_*\mathcal O_{X^\sharp}(sP^\sharp +ks\pi^*\phi^*H) 
\hookrightarrow \mathcal O_{X^\flat}(sP^\flat +ks\rho^*\phi^*H)
\label{mor:inj}
\end{align}
is an isomorphism over $U^\flat$. 
Note that $\rho$ is an isomorphism over $U$. 
Hence, we show that 
$
\left( \sigma_*\mathcal O_{X^\sharp}(sP^\sharp +ks\pi^*\phi^*H) \right)|_T
$
is globally generated over $U^\flat$. 
Put $\mathcal M :=\left( \sigma_*\mathcal O_{X^\sharp}(P^\sharp) \right)|_T$. 
By Step~\ref{step:DM}, we see that 
\begin{align*}
\mathcal G:=
\left( h_* 
\mathcal M 
\right)^{\otimes s}
\otimes \mathcal O_{W}(ks\tau^*H|_W)
\cong
\left.  \left( 
\left( \phi^\sharp_*\mathcal O_{X^\sharp}(P^\sharp) \right)^{[s]}
\otimes \mathcal O_{Y^\flat}(ks\tau^*H)
\right) \right|_W
\end{align*}
is globally generated over $V^\flat$. 
Since 
$
\sigma_*\mathcal O_{X^\sharp}(P^\sharp) 
\overset{\alpha}{\hookrightarrow} \mathcal O_{X^\flat}(P^\flat) 
$
is an isomorphism on $U^\flat$ and 
$P^\flat|_{U^\flat}$ is a Cartier divisor that is relatively very ample over $V^\flat$, 
we see that $\mathcal M|_{U^\flat}$ is a line bundle that is relatively very ample over $V^\flat$, 
so the natural morphisms
\begin{align}
h^* \left( h_* \mathcal M \right)^{\otimes s}
\xrightarrow{} \mathcal M^{\otimes s} 
\xrightarrow{\alpha^{\otimes s}|_T} \mathcal O_T(s P^\flat|_T)
\label{mor:n}
\end{align}
are surjective over $U^\flat$, which induce the morphism  
\begin{align}
h^*\mathcal G
\to \mathcal O_T(sP^\flat|_T+ksh^*(\tau^*H)|_W)
\end{align}
that is surjective over $U^\flat$, 
and thus the claim of this step follows from 
$h^*(\tau^*H)|_W=(\rho^*\phi^*H)|_T$.  
}
%
\end{step}
\begin{step} \label{step:last}
We complete the proof. By the projection formula, we have 
\begin{align*}
\mathcal O_X(ks(K_{X/Y}+\Delta+\phi^*H))
\cong \rho_*\mathcal O_{X^\flat}(sP+ks\rho^*\phi^*H). 
\end{align*}
By Step~\ref{step:free}, the right-hand side is globally generated over $U\supseteq X_v$, 
since $\rho$ is an isomorphism over $U$ and $\rho_*\mathcal O_{X^\flat}\cong \mathcal O_X$. 
Thus $\mathbb B(K_{X/Y}+\Delta+\phi^*H)\cap X_v=\emptyset$. 
\end{step}
\end{proof}
\begin{proof}[Proof of Theorem~\ref{thm:inclusions}.]
Put $L:=K_{X/Y}+\Delta$. 
We first prove (1). 
Under the assumption of $\mathbb B_-(D-L)\cap X_v=\emptyset$, 
we show that $\mathbb B_-(D)\cap X_v =\emptyset$. 
Let $A$ be an ample $\mathbb Q$-Cartier divisor on $X$ 
{such that $\mathbb B(D-L+A)\cap X_v=\emptyset$.}
It is enough to show that 
$$\mathbb B(D+2A)\cap X_v =\emptyset.$$
{By the choice of $A$,}
we can find a $\mathbb Q$-Cartier divisor $\Gamma\ge0$ on $X$ such that
\begin{itemize}
\item $\Gamma \sim_{\mathbb Q}D-L+A$, and
\item $V(\phi,\Delta')$ defined in Notation~\ref{notation:mor} contains $v$, where $\Delta':=\Delta+\Gamma$.
\end{itemize}
By the assumption, shrinking $V$ if necessary, we may assume that $(A+D)|_U$ is ample over $V$. 
%
It follows from 
$$ K_{X/Y}+\Delta' \sim_{\mathbb Q} L+\Gamma \sim_{\mathbb Q} A+D  $$
that $K_{U/V}+\Delta'|_U$ is also ample over $V$. 
Take an ample $\mathbb Q$-Cartier divisor $H$ on $Y$ so that $A-f^*H$ is ample. 
Then Proposition~\ref{prop:base locus} implies that  
$$
\mathbb B(2A+D) \cap X_v
\subseteq \mathbb B(A+D+\phi^*H) \cap X_v
=\mathbb B(K_{X/Y}+\Delta'+ \phi^*H) \cap X_v
=\emptyset. 
$$
Note that $X_v\subseteq U$. 

Now we prove (2). 
Assuming $\mathbb B_+(D-L)\cap X_v=\emptyset$, we show $\mathbb B_+(D)\cap X_v =\emptyset$. 
Take an ample $\mathbb Q$-Cartier divisor $A$ on $X$ so that  
\begin{itemize}
\item $\mathbb B_+(D-L)=\mathbb B(D-L-A) \supseteq \mathbb B_-(D-L-A)$, and  
\item $\mathcal O_{X_v}(j(D-A))$ is nef for some $j\ge1$ divisible enough. 
\end{itemize}
Then $\mathbb B_-(D-L-A) \cap X_v =\emptyset$ by the assumption, 
and thus $\mathbb B_-(D-A)\cap X_v=\emptyset$ by (1). 
Then the assertion follows from 
$\mathbb B_+(D)\subseteq \mathbb B_-(D-A)$. 
\end{proof}
\begin{cor}\label{cor:pullback}
Let the notation be as in Notation \ref{notation:mor}. 
Let $E$ be a $\mathbb Q$-Cartier divisor on $Y$. 
Then we have: 
\begin{itemize}
	\item[(1)] $V(\phi,\Delta) \cap \mathbb B_-(E) 
\subseteq 
\phi\left(\mathbb B_-(\phi^*E -(K_{X/Y}+\Delta))\right) $, and 
\label{inclusion:cor-1}
\item[(2)]
$V(\phi,\Delta)\cap \mathbb B_+(E) 
\subseteq 
\phi\left(\mathbb B_+(\phi^*E -(K_{X/Y}+\Delta))\right)$.
\label{inclusion:cor-2}
\end{itemize}
\end{cor}
\begin{rem} \label{rem:FG, Deng}
By putting $E:=-K_Y$ in the above corollary, we can recover 
\cite[Corollary~2.9]{KoMM92} (in characteristic zero), 
\cite[Theorems~D and~E]{Den17}, and \cite[Theorem~1.1]{FG12}. 
\end{rem}
\begin{proof}[Proof of Corollary~\ref{cor:pullback}]
Put $V:=V(\phi,\Delta)$. 
Since we have $V\cap \mathbb B_-(E) = V\cap \phi(\mathbb B_-(\phi^*E))$ by Lemma~\ref{lem:pullback of base loci}~(1), we get (1) of the corollary 
by applying Theorem~\ref{thm:inclusions}~(1) for $D=\phi^*E$.
We show (2). Put $L:=K_{X/Y}+\Delta$. 
Pick an ample $\mathbb Q$-Cartier divisor $A$ on $X$ such that 
$\mathbb B_+(\phi^*E -L) =\mathbb B(\phi^*E -L -A)$. 
Take an ample $\mathbb Q$-Cartier divisor $H$ on $Y$ so that 
$\mathbb B_+(E) \subseteq \mathbb B_-(E -H)$ 
and $A-\phi^*H$ is ample. 
Then we have 
$$
V\cap \mathbb B_+(E) 
\subseteq 
V\cap \mathbb B_-(E-H) 
\overset{\textup{(1)}}
{\subseteq} 
V\cap \phi \left( \mathbb B(\phi^*E-\phi^*H -L) \right),
$$
and thus the assertion follows from 
$\mathbb B(\phi^*E-\phi^*H-L) \subseteq \mathbb B(\phi^*E -L -A) = \mathbb B_+(\phi^*E -L)$. 
\end{proof}
\begin{cor}\label{cor:dominate}
Let the notation be as in Notation \ref{notation:mor}. 
Then $\mathbb B_+(-(K_{X/Y}+\Delta))$ and $\mathbb B_-(-(K_{X/Y}+\Delta+\phi^*H))$
are dominant over $Y$ for any ample $\mathbb Q$-Cartier divisor $H$ on $Y$ 
unless $Y$ is one point. 
\end{cor}
\begin{proof}
This follows from Corollary~\ref{cor:pullback} by putting $E=0$ or $E=-H$. 
\end{proof}
\begin{cor} \label{cor:rel min model}
Let the notation be as in Notation \ref{notation:mor}. 
Suppose that $\mathcal O_{X_v}(i(K_X+\Delta))$ is nef for $i \in \mathbb{Z}_{>0}$ divisible enough and a point $v\in V(\phi,\Delta)$.  
Then $\mathbb B_-(K_{X/Y}+\Delta) \cap X_v = \emptyset$. 
In particular, $K_{X/Y}+\Delta$ is pseudo-effective. 
\end{cor}
\begin{proof}
Set $D:=K_{X/Y}+\Delta$.  Then we have $\mathbb B_-\left(D-(K_{X/Y}+\Delta)\right)=\emptyset$, and thus the assertion 
follows from Theorem~\ref{thm:inclusions}~(1). 
\end{proof}

\subsection{Algebraic proof of Campana--Cao--Matsumura's equality}
The following theorem is a slight extension of  \cite[Theorem 1.2]{CCM}, 
which generalizes Hacon--M\textsuperscript cKernan's question proved in \cite{EG19}. 
In this subsection, we give an algebraic proof for this equality.

\begin{thm}\label{thm:nd}
Let $X$ be a normal projective variety and let $\Delta$ be an effective $\mathbb Q$-divisor on $X$ such that $(X,\Delta)$ is klt. 
Let $\phi:X\to Y$ be a morphism with connected fibers onto
a $\mathbb Q$-Gorenstein normal projective variety $Y$ 
with only canonical singularities. 
If $-(K_{X/Y}+\Delta)$ is nef, then 
\begin{align*}
\mathrm{nd}(-(K_{X/Y}+\Delta))=\mathrm{nd}(-(K_{F}+\Delta|_F)) 
\text{ for a general fiber } F.  
\end{align*}
\end{thm}
To prove this theorem, we need the following proposition:
\begin{prop}\label{prop:univ bound}
Let the notation be as in Notation \ref{notation:mor}. 
Let $V$ be a Zariski-open subset of $V(\phi,\Delta)$ and put $U:=\phi^{-1}(V)$. 
Suppose that 
\begin{itemize}
\item $(U,\Delta|_U)$ is klt, and
\item $-(K_{X/Y}+\Delta)$ is nef. 
\end{itemize}
Let $D$ be a $\mathbb Q$-Cartier divisor on $X$ such that 
$D|_U$ is relatively nef over $V$. 
Fix a smooth closed subvariety $Z\subset Y$ that is properly contained in $V$. 
Then there exists a constant $C$ such that  
$$
{\mathrm{mult}}_{\phi^{-1}(Z)}(\Gamma) \leq C
$$
for every $\mathbb Q$-Cartier divisor $\Gamma$ with 
$0\le \Gamma \equiv D-\lambda (K_{X/Y}+\Delta)$ 
for some $0 \le \lambda \in \mathbb Q$. 
Here $\mathrm{mult}_{\phi^{-1}(Z)}(\Gamma)$ 
is the multiplicity of $\Gamma$ along $\phi^{-1}(Z)$. 
\end{prop}
Note that it follows from the definition of $V(\phi,\Delta)$ that 
$\phi^{-1}(Z)$ is irreducible. 
\begin{proof}
Let $Y'$ (resp. $X'$) be the blow-up of $Y$ (resp. $X$) along $Z$ (resp. $\phi^{-1}(Z)$), and let $\phi':X'\to Y'$ denote the induced morphism. 
Then we have the following cartesian diagram:
\begin{align*}
\xymatrix{ X \ar[d]_-\phi & X' \ar[l]_-\sigma \ar[d]^-{\phi'}  \\ 
Y & Y'. \ar[l]^-\tau
}
\end{align*}
It is easy to check that the following hold:
\begin{itemize}
\item $\sigma^*(K_{X/Y}+\Delta) \sim K_{X'/Y'}+\Delta'$, where $\Delta':=\sigma^{-1}_*\Delta$. 
\item $(\sigma^{-1}(U),\Delta'|_{\sigma^{-1}(U)})$ and $\tau^{-1}(V)$ 
respectively satisfy the conditions on $(U,\Delta|_U)$ and $V$ in Theorem~\ref{thm:inclusions}.
\end{itemize}
Hence by respectively replacing $\phi:X\to Y$ and $D$ 
with $\phi':X'\to Y'$ and $\sigma^*D$, 
we may assume that $Z$ is a smooth prime divisor on $Y$. 

Put $L:=K_{X/Y}+\Delta$. 
Take an ample $\mathbb Q$-Cartier divisor $A$ on $X$. 
Fix $\lambda\in\mathbb Q_{\ge0}$ and $0\le\Gamma\equiv D-\lambda L$. 
Then $\Gamma$ can be written as 
$\Gamma=\alpha \phi^*Z +E$ with ${\mathrm{mult}}_{\phi^*Z}(E)=0$, 
{where $\alpha:=\mathrm{mult}_{\phi^*Z}(\Gamma)$.}
Since $(U,\Delta|_U)$ is klt, 
the pair $\big(U,(\Delta +(\lambda+m)^{-1}E)|_U \big)$ is also klt for some $m\gg0$. 
Since $A-mL$ is ample, there is $0\le G \sim_{\mathbb Q} A-mL$ such that $\left(U,B|_U\right)$ is klt,
where $B:=\Delta+(\lambda+m)^{-1}(E+G)$. 
Then we have 
\begin{align*}
D-\lambda(K_{X/Y}+\Delta) &\equiv \alpha \phi^*Z +E 
\textup{\quad and \quad} 
A-m(K_{X/Y}+\Delta) \equiv G, 
\end{align*}
and thus we can see that 
\begin{align*}
\frac{1}{\lambda+m}(A+D-\alpha \phi^*Z)
\equiv K_{X/Y}+\Delta+\frac{1}{\lambda+m}(E+G) 
= K_{X/Y}+B. 
\end{align*}
The divisor $\left(K_{X/Y}+B\right)|_U$ is relatively nef over $V$ 
by the assumption, 
and thus it follows that  
$K_{X/Y} +B$ is pseudo-effective from Corollary~\ref{cor:rel min model}. 
This implies that $A+D-\alpha \phi^*Z$ is also pseudo-effective. 
Then we obtain  
\begin{align}
C:=\frac{(A+D)\cdot A^{\cdot \dim X-1}}{\phi^*Z\cdot A^{\cdot \dim X-1}} 
\ge \alpha={\mathrm{mult}_{\phi^*Z}}(\Gamma),
\label{ineq:mult}
\end{align}
which completes the proof. 
\end{proof}
\begin{proof}[Proof of Theorem~\ref{thm:nd}]
The inequality ``$\ge$'' follows from the well-known fact that the numerical Kodaira dimension of a nef $\mathbb Q$-Cartier divisor $D$ on a projective variety $V$ coincides with 
$$
{\mathrm{max}}\left\{\dim W 
\middle| \textup{closed subvariety $W\subseteq V$ with $W\cdot D^{\cdot \dim W} >0$}
\right\}. 
$$
We prove the converse inequality ``$\le$''.  Put $L:=K_{X/Y}+\Delta$. 
Take an ample Cartier divisor $A$ on $X$. 
Let $\mathcal I$ denote the ideal defining $F$ 
and set $\mathcal G_m:=\mathcal I^m/\mathcal I^{m+1}$ 
for each $m\in\mathbb Z_{\ge0}$. 
By the exact sequence
\begin{align*}
0\to \mathcal I^{m+1}(A-lL) \to \mathcal I^m(A-lL) \to \mathcal G_m(A-lL) \to 0, 
\end{align*}
we get 
\begin{align}
h^0(X, \mathcal I^m(A-lL)) 
\le 
h^0(X, \mathcal I^{m+1}(A-lL)) 
+
h^0(F, \mathcal G_m (A-lL)). 
\label{ineq:1}
\end{align} 
Thanks to Proposition~\ref{prop:univ bound}, 
we can find $m_0\in\mathbb Z_{>0}$ such that
$h^0(X,\mathcal I^{m_0}(A-lL)) =0$ for each $l\in\mathbb Z_{\ge0}$, 
and thus from (\ref{ineq:1}) we obtain that 
\begin{align*}
h^0(X, \mathcal O_X(A-lL)) 
\le 
\sum_{0\le m < m_0} h^0(F, \mathcal G_m (A-lL)) 
\end{align*}
for each $l\in\mathbb Z_{\ge0}$. 
Put $\nu:=\mathrm{nd}(-L|_F)$.
Since each $\mathcal G_m$ is a torsion-free $\mathcal O_F$-module, 
it is isomorphic to a subsheaf of the direct sum of ample divisors, 
and thus there is a constant $C$ such that 
\begin{align*}
\frac{h^0(X, \mathcal O_X(A-lL))}{l^\nu}
\le 
\sum_{0\le m < m_0} \frac{h^0(F, \mathcal G_m (A-lL)) }{l^\nu}
\le C. 
\end{align*}
This means that $\mathrm{nd}(-L) \le \nu=\mathrm{nd}(-L|_F)$.
\end{proof}

\section{On asymptotic variants of base loci} \label{section:small B-}

The aim of this section is to prove Theorem \ref{thm:2_intro}. 
For this purpose, we need 
Proposition \ref{r-c10}, Theorem \ref{r-viehweg}, and Theorem \ref{r-thm:locallytrivial},  
which have been essentially proved in the case $-(K_{X/Y}+\Delta)$ 
being nef by \cite{CH19} and \cite{CCM}, 
but we actually need them under the weaker assumption 
that its restricted base locus is not dominant over $Y$.  
In subsection \ref{subsection:semistability}, for reader's convenience, 
we give a proof and an explanation for them with slightly generalized form.

Throughout this section, let $\phi : X \rightarrow Y$ be 
a surjective morphism with connected fiber between projective manifolds. 
Let $\Delta$ be an effective $\mathbb{Q}$-divisor on $X$ 
and fix $N \in \mathbb{Z}_{>0}$ such that $N \Delta$ is a Cartier divisor.

\subsection{Extension of Cao-H\"oring's work for nef anti-canonical divisors} \label{subsection:semistability}

We first begin with the following proposition:

\begin{prop} \label{flatness}
If $(X,  \Delta)$ is an lc pair and 
$\mathbb{B}_{-}(- (K_{X/Y} + \Delta))$ is not dominant over $Y$, 
then $\phi$ is flat and semistable 
$($that is, the fiber $\phi^{-1}(y)$ has the same dimension and it is reduced  for any $y \in Y$$)$. 
\end{prop}
\begin{proof}[Proof of flatness]
Our proof follows an argument in \cite{PZ19}  
and the basic idea is the same as in \cite{LTZZ}. 
Let $y\in Y$ be a point such that $X_y$ has the largest dimension among all the closed fibers. 
If $\mathrm{codim}_X(X_y)=0$, then $Y=\{y\}$, and thus $f$ is flat. 
We may assume that $\mathrm{codim}_X(X_y)\ge1$. 
\setcounter{step}{0}
\begin{step}
We first prove the assertion in the case where $\mathrm{codim}_X(X_y)=1$. 
Let $\rho:Y^\flat\to Y$ be a resolution passing through the flattening of $\phi$. 
Let $X^\flat$ be a resolution of the main component of $X\times_Y Y^\flat$. 
We have the following commutative diagram:
\begin{align*}
\xymatrix{
X^\flat \ar[r]^-\pi \ar[d]_-{\phi^\flat} & X \ar[d]^\phi \\
Y^\flat \ar[r]^-\rho & Y.
}
\end{align*}
Then we have: 
\begin{itemize}
\item all ${\phi^\flat}$-exceptional divisors are ${\pi}$-exceptional,
\item the $\rho$-exceptional divisor $D:=K_{Y^{\flat}}-{\rho}^{*}K_{Y}$ on $Y^\flat$ is effective, and 
\item 
$\mathbb{B}_{-}(-{\pi}^{*}(K_{X/Y}  +  \Delta))$
is not dominant over $Y^{\flat}$ 
\end{itemize}
Set $E:=K_{X^{\flat}} +\pi^{-1}_*\Delta -{\pi}^{*}(K_X+\Delta) $.
Then $E$ is $\pi$-exceptional, and 
$$
K_{X^{\flat}/Y^{\flat}}+\pi^{-1}_*\Delta -{\pi}^{*}(K_{X/Y}+\Delta) 
=E -{\phi^\flat}^{*}D.
$$

We show that $K_{X^{\flat}/Y^{\flat}}+\pi^{-1}_*\Delta -{\pi}^{*}(K_{X/Y}+\Delta) $ is pseudo-effective. 
Set $L :=-{\pi}^{*}(K_{X/Y}+\Delta)$ and $M:=K_{X^\flat/Y^\flat}+\pi^{-1}_*\Delta +L$. 
Let $F$ be a very general fiber of $\phi^{\flat}$. 
Then $\mathbb B_-(L)$ does not intersect with $F$.
Hence Theorem~\ref{thm:inclusions}~(1) implies that $\mathbb B_-(M)$ does not intersect $F$, and thus $M$ is pseudo-effective. 

Let $A$ be a very ample divisor on $X$ and $\alpha := {\pi}^{*}A^{\dim X -1}$. 
Since $E$ is ${\pi}$-exceptional, we have
\begin{equation*}
0 \ge -(E -{\phi^\flat}^{*}D) \cdot \alpha =({\phi^\flat}^{*}D)\cdot\alpha.
\end{equation*}
On the other hand, if $\phi$ is not flat at $y$, we see that $\rho(\mathrm{Supp}(D))\ni y$, 
and thus ${\pi}_*{\phi^\flat}^*D \ge (X_y)_{\mathrm{red}}$. 
Then we get $({\phi^\flat}^{*}D)\cdot\alpha >0$, which is a contradiction.

\begin{rem}
The argument in this step works when $K_X+\Delta$ is $\mathbb Q$-Cartier and the non-canonical locus of $(X,\Delta)$ is not dominant over $Y$. 
\end{rem}

\end{step}
\begin{step}
Take $r\ge 2$. 
Assuming to the assertion holds if $\mathrm{codim}_X(X_y)<r$, 
we prove it in the case when $\mathrm{codim}_X(X_y)=r$. 
Let $Z\subseteq Y$ be the closed subset consisting of points over which $\phi$ is not flat, 
and $Z'\subseteq Z$ is the Zariski-closure of $Z\setminus\{y\}$
Let $H_0$ be a very ample divisor on $Y$
and $\mathfrak d \subseteq |H_0|$ be the linear system of members containing $y$.
Take a very general member $H\in\mathfrak d$. 
Then $H$ is smooth and $H\cap Z''$ is properly contained in $Z''$ for 
each irreducible component $Z''$ of $Z'$.
This implies that $\phi^{-1}(H)\cap W $ is also properly contained in $W$ 
for each irreducible component $W$ of $\phi^{-1}(Z')$, 
and thus we have 
$$
\dim (\phi^{-1}(H\cap Z'))
\le\dim \phi^{-1}(Z') -1
\le\dim X-2.
$$
Then it follows that $\dim (\phi^{-1}(H\cap Z))\le \dim X-2$ 
from $Z=Z'\cup\{y\}$ and $\mathrm{codim}_X(X_y)\ge2$. 
Let $G$ be an irreducible component of $\phi^{-1}(H)$.
Then $\dim G=\dim X-1$, and thus $G\not\subseteq \phi^{-1}(Z)$.
Hence $\phi$ is flat at a point in $G$,
and thus so is the induced morphism $\psi:\phi^{-1}(H)\to H$.
In particular $G$ dominates $H$, 
from which we see that $\phi^{-1}(H)$ is irreducible, 
since a general fiber of $\psi$ is a smooth variety. 
We also see that $\phi^*H=G$ as divisors. 
Let $\nu:X'\to G$ be the normalization and 
$C\ge0$ be the Weil divisor on $X'$ corresponding the conductor ideal. 
Then $\mathrm{Supp}(C)$ is not dominant over $Y$. 
Put $\Delta':=C+\nu^*(\Delta|_G)$. 
Then we have 
\begin{align*}
K_{X'/H}+\Delta'
& \sim_{\mathbb Q} \nu^*(K_G +\Delta|_G -\psi^*K_H) 
\\ & \sim_{\mathbb Q} \nu^*\left(K_X+G+\Delta -\phi^*(K_Y+H)\right) |_G
\sim_{\mathbb Q} \nu^*\left(K_{X/Y}+\Delta\right)|_G, 
\end{align*}
and thus we can see that $\phi\left(\mathbb B_-(-(K_{X'/H}+\Delta'))\right)$ 
is properly contained in $ H$. 
Since $\mathfrak d$ is free on $Y\setminus\{y\}$, 
the non-lc locus of $(G, \Delta|_G)$ is not dominant over $Y$, 
so the same property holds for $(X', \Delta')$. 
Then, by the assumption, $X'$ is flat over $H$, we obtain 
$$
\mathrm{codim}_X(X_y)-1
=\mathrm{codim}_{X'}(X'_y) 
=\dim H
=\dim Y -1, 
$$ 
and thus $\mathrm{codim}_X(X_y)=\dim Y$. This means that $\phi$ is flat. 
\end{step}
\end{proof}

\begin{proof}[Proof of semistability]
Assume there exists a point $y\in Y$ such that $\phi^{-1}(y)$ is not reduced. 
We can take a smooth curve $C$ in $Y$ such that 
$y \in C$ and $\mathbb{B}_{-} (-(K_{X/Y} + \Delta) )$ is not dominant over $C$.
Set $T:=\phi^{-1}(C) $ and $\Delta_T := \Delta |_{T}$.
By the assumption, 
there exists a prime divisor $W$ on $T$ such that $e : = \text{ord}_{W}({\phi|_{T}}^{*}(y)) >1$.
Take $\pi^\sharp : T^\sharp \rightarrow T$ be a log resolution of $(T,\Delta_T)$ and 
set $\phi^\sharp : = \phi|_{T} \circ \pi^\sharp $. 
Then we have the following diagram: 
\begin{equation*}
\xymatrix@C=40pt@R=30pt{
T^{\sharp}\ar[r]^{\pi^\sharp} \ar[d]^{\phi^\sharp }&T\ar@{}[r]|*{\subseteq} \ar[d]^{\phi|_{T}}& X  \ar[d]^{\phi} \\  
C\ar@{=}[r]  & C \ar@{}[r]|*{\subseteq}  & Y. 
}
\end{equation*}
Let $\Delta^{\sharp} $ and $G$ be the effective divisors  
such that 
$K_{T^{\sharp}} + \Delta^{\sharp} = {\pi^\sharp}^{*}(K_{T} + \Delta_T) +G$.
Then $(T^{\sharp},\Delta^{\sharp} )$ is log canonical, 
the divisor $G$ is $\pi^\sharp$-exceptional,  
and $\mathbb{B}_{-}( -{\pi^{\sharp} }^{*}(K_{T/C} + \Delta_T) )$ is not dominant over $C$. 

Now we define an $mN$-Bergman metric $H_{mN}$ on $mNG+2A^{\sharp}$ 
for any $m \in \mathbb{Z}_{>0}$ and 
an ample divisor $A^{\sharp}$ on $T^{\sharp}$ such that $N \Delta^{\sharp} +A^\sharp$ is ample. 
We consider the equality 
\begin{align*}
&mNG+2A^{\sharp} \\
=& mNK_{T^{\sharp}/C}+\overbrace{(-mN{\pi^{\sharp} }^{*}(K_{T/C} +\Delta_T)  + A^{\sharp})}^{\textup{with $h_1$}}+
\overbrace{(N\Delta^{\sharp} + A^{\sharp})}^{\textup{with $h_2$}}+
\overbrace{mN(1-1/m)\Delta^{\sharp}}^{\textup{with $h_3^{mN}$}}.
\end{align*}
and the singular hermitian metrics $h_1$, $h_2$, and $h_3$ with semipositive curvature 
satisfying the following: 
\begin{itemize} 
\item  $h_1$ is a singular hermitian metric on $-mN{\pi^{\sharp} }^{*}(K_{T/C} + \Delta_T)  + A^{\sharp}$ 
such that $h_1|_{T^{\sharp}_{w}}$ is smooth on $T^{\sharp}_{w}$ for a general point $w \in C$, 
which is obtained from the property that 
$\mathbb{B}_{-}( -{\pi^{\sharp} }^{*}(K_{T/C} + \Delta_T) )$ is not dominant over $C$. 

\item $h_2$ is a smooth hermitian metric on the ample divisor $ N \Delta^{\sharp} +A^\sharp$ 
with positive curvature.  

\item $h_3$ is a singular hermitian metric on $(1-1/m)\Delta^{\sharp}$ 
such that $\mathcal{J}(h_3 |_{T^{\sharp}_{w}}) = \mathcal{O}_{T^{\sharp}_{w}}$ holds 
for a general point $w \in C$, which is obtained from the klt condition of $(T,  (1-1/m)\Delta^{\sharp})$. 
\end{itemize}


Then the singular hermitian metric $h := h_{1} h_{2} h_{3}^{mN}$ has semipositive curvature current and 
satisfies that $\mathcal{J}(h^{\frac{1}{mN}} |_{T^{\sharp}_{w}}) = \mathcal{O}_{T^{\sharp}_{w}}$ 
for a general point $w \in C$.  
Then, by ${\phi^\sharp}_{*}(2A^{\sharp} + mNG) \neq 0$ and the above properties, 
we can see that the induced $mN$-Bergman metric $H_{mN}$ on $2A^{\sharp} + mNG$ 
has  semipositive curvature (see \cite[A.2]{BP}). 
Then, by \cite[Remark 2.5]{CP17} or \cite[Proposition 2.5]{Wang},
we obtain 
$$
\sqrt{-1}\Theta_{H_{mN}} \ge mN(e-1)[{\pi^{\sharp}}^{-1}_{*}W].
$$
This means that $(2A^{\sharp} + mNG) - mN(e-1){\pi^{\sharp}}^{-1}_{*}W$ is pseudo-effective 
for any $m \in \mathbb{Z}_{>0}$, 
and thus so is $G-(e-1){\pi^{\sharp}}^{-1}_{*}W$. 
This is a contradiction to the fact that $G$ is $\pi^\sharp$-exceptional if $e>1$. 
\end{proof}

The proof of the following proposition is essentially the same as in \cite[Section 4.3]{CCM}. 

\begin{prop}$($cf. \cite[Section 4.3]{CCM}$)$
\label{nonvert}
Under the same assumption as in Proposition \ref{flatness}, 
the vertical part of $\Delta$ is 0.
\end{prop}
\begin{proof}
Let $\Delta^{v}$ (resp. $\Delta^{h}$) be the vertical part (resp. the horizontal part) of $\Delta$.
We take an ample divisor $A$ of $X$ such that  $N\Delta^{h} +A$ is ample.
It is sufficient to prove that $-mN\Delta^{v} +2A$
is pseudo-effective for any $m \in \mathbb{Z}_{>0}$.
By the same way as in the proof of Proposition \ref{flatness}, 
we can see that $-mN\Delta^{v} +2A$ is pseudo-effective 
since we have 
$$
-mN\Delta^{v} +2A = mNK_{X/Y} +(-mN(K_{X/Y} + \Delta) +A )+(N\Delta^{h} +A) + mN(1-1/m)\Delta^{h}. 
$$

\end{proof}

As we confirmed in the proof of Proposition \ref{flatness}, 
the assumption of $-(K_{X/Y}+\Delta)$ being nef  
can be replaced with our weaker assumption that 
$\mathbb{B}_{-}(- (K_{X/Y} + \Delta) )$ is not dominant over $Y$ 
in the situation where we use positivity of direct images. 
Hence, by the same proof as in \cite{CH19, CCM}, 
we can obtain Proposition \ref{r-c10}, Theorem \ref{r-viehweg} and \ref{r-thm:locallytrivial} 
by paying attention to the flatness and semistability of $\phi$. 
For this reason, we put the proof of them in Appendix.

\begin{prop}[{cf. \cite[2.8 Proposition]{CH19}, \cite[Theorem 2.2(1)]{CCM}}]
\label{r-c10}
Let $L$ be a $\phi$-big divisor on $X$ 
admitting a singular hermitian metric $h_L$ 
such that $\sqrt{-1}\Theta_{h_L}\ge \phi^{*}\theta$ 
for some smooth $(1,1)$-form $\theta$ on $Y$. 
If $(X, \Delta)$ is a klt pair and 
$\mathbb{B}_{-}(- (K_{X/Y} + \Delta))$ is not dominant over $Y$, 
then for any $\varepsilon >0$ and any $p,q \in \mathbb{Z}$ 
with $\phi_{*}(pK_{X/Y} + q N\Delta +L)\neq 0$,
there exists a singular hermitian metric $H$ on $\phi_{*}(pK_{X/Y} + q N\Delta +L)$ such that
$$
\sqrt{-1}\Theta_{H} \succeq -\varepsilon\omega_{Y} \otimes \id +(1-\varepsilon ) \theta \otimes \id.
$$
Here $\omega_{Y}$ is a K\"ahler form on $Y$.

In particular, if $L$ is a pseudo-effective and $\phi$-big divisor of $X$, 
then $\phi_{*}(pK_{X/Y} + q N\Delta +L)$ is weakly positively curved.
\end{prop}

\begin{theo}[\text{cf. \cite[3.5 Lemma]{CH19}, \cite[Proposition 3.15]{Cao19}, \cite[Proposition 3.5]{CCM}}]
\label{r-viehweg}
Let $L$ be a $\phi$-big divisor on $X$ and $r$ be the rank of $\phi_{*}L$.
If $(X, \Delta)$ is a klt pair and 
$\mathbb{B}_{-}(- (K_{X/Y} + \Delta))$ is not dominant over $Y$, 
then there exists a $\phi$-exceptional effective divisor $E$ on $X$ such that 
$$L + E - \frac{1}{r}\phi^{*}(\det (\phi_{*}L))$$
is pseudo-effective.
Moreover $E$ can be taken as the zero divisor if $\phi_{*}L$ is locally free.
\end{theo}

\begin{theo}[\text{cf. \cite[Theorem 1.3]{CCM}, \cite[Remarque 2.10]{Cao}}]

\label{r-thm:locallytrivial}
If $(X, \Delta)$ is a klt pair and
$\mathbb{B}_{-}(- (K_{X/Y} + \Delta))$ is not dominant over $Y$, 
then $\phi$ is locally trivial with respect to $(X, \Delta)$.
Further the fiber product 
$(\tilde{X},  \tilde{\Delta}) := (X,\Delta) \times_{Y} \tilde{Y}$
by the universal cover $\pi : \tilde{Y} \rightarrow Y$ 
admits the following splitting:
$$
(\tilde{X},  \tilde{\Delta})  \cong \tilde{Y} \times (F, \Delta|_{F})
$$
where $F$ is the fibre of $\phi$. 
Moreover there exists a homomorphism $\rho: \pi_1(Y) \to \Aut{F}$ 
such that $X$ is isomorphic to the quotient of $\tilde{Y} \times F $ 
by the action of $\pi_1(Y)$ and $\rho$. 
\end{theo}


\subsection{Variants of base loci and flatness of direct images}
In this subsection, we prove Theorem \ref{thm:nefness}, \ref{thm:nonLelongnumber}, and \ref{thm:semiampleness}, which lead to Theorem \ref{thm:2_intro}. 
In \cite{Cao19, CH19, CCM},  
it was proven that nefness of $-K_{X/Y}$ implies 
numerical flatness of the direct image sheaves, 
however, in this subsection, we study the converse implication. 
This paper reveals a relation between positivity conditions 
(nefness, semi-ampleness, vanishing Lelong number) and 
certain flatness of the direct image sheaves.

\begin{theo}
\label{thm:nefness}
If $\mathbb{B}_{-}(- K_{X/Y} )$ is not dominant over $Y$, then it is empty $($that is, $-K_{X/Y}$ is nef$)$. 

\end{theo}
\begin{proof}
Note that $\phi: X\to Y$ is automatically locally trivial  
by Theorem \ref{r-thm:locallytrivial}. 
For a sufficiently ample divisor $A$ on $X$, we define the divisor $\tilde{A}$ on $X$ 
by 
$$
\tilde{A}:= A  - \frac{1}{r}\phi^{*}\det \phi_{*}A
$$
by the same way as in Theorem \ref{r-viehweg}. 
Also, for every  $m \in \mathbb{Z}_{>0}$, we define the coherent sheaf $\mathcal{V}_m$ by 
$$
\mathcal{V}_m:= \phi_*(-mK_{X/Y}+ \tilde{A}). 
$$
The sheaf $\mathcal{V}_m $ is reflexive by the flatness of $\phi$, 
and thus it is actually locally free and numerically flat by 
Theorem \ref{flatness}, \ref{r-c10}, 
and \cite[Proposition 2.7]{CCM} (see also \cite[Theorem 3.4]{HIM19}). 
Then we find an ample divisor $B$ (independent of $m$) 
such that  $\mathcal{V}_m \otimes B$ is globally generated on $X$,  
by the argument of the Castelnuovo-Mumford regularity. 
Indeed, the Kodaira type vanishing theorem holds for numerically flat vector bundles by Lemma \cite[Lemma 2.10]{CCM}, 
and thus $\mathcal{V}_m$ is $(n+1)$-regular with respect to $A_Y$, namely, 
$H^{q}(Y, \mathcal{V}_m \otimes A_Y^{\otimes (n+1-q)})=0$ holds for any $q>0$. 
Here $A_Y$  is a very ample divisor on $Y$ such that $A_Y -K_Y$ is ample and $n$ is the dimension of $Y$. 
This implies that $B:=A_Y^{\otimes n+1}$ satisfies the desired property (for example see \cite[Theorem 1.8.5]{Laz}). 

For an arbitrary point $y \in Y$, the natural map 
$$
H^{0}(X, -mK_{X/Y} + \tilde{A} + \phi^*B)=H^{0}(Y, \mathcal{V}_m\otimes B) 
\xrightarrow{\quad \quad }
(\mathcal{V}_m\otimes B)_y
$$
is surjective since $\mathcal{V}_m \otimes B$ is globally generated. 
On the other hand, there exists $y \in Y$ such that $-K_{X_y}=-K_{X/Y}|_{Y_y}$ is nef 
since $\mathbb{B}_-(-K_{X/Y})$ is properly contained in $Y$. 
Then if follows that $K_{X_y}$ is nef for any $y\in Y$
since all the fibers of $\phi$ are isomorphic each other by Theorem \ref{r-thm:locallytrivial}. 
Hence we may assume that $-mK_{X_y} + A|_{X_y}$ is globally generated 
by choosing an sufficiently ample divisor $A$ again.   
Then the evaluation map 
$$
H^{0}(X_y, -mK_{K_y} + A|_{X_y} ) =
H^{0}(X_y, -mK_{X_y} + (\tilde{A} + \phi^*B) |_{X_y})
\to (-mK_{K_y} + \tilde{A} + \phi^*B|_{X_y})_x
$$
is surjective for any $m\in  \mathbb{Z}_{>0}$ and any $x \in X_y$. 
The above surjective maps imply that $-mK_{X/Y} + \tilde{A} + \phi^*B$ is globally generated for any $m$. 
We  can conclude that $-K_{X/Y}$ is nef as $m$ tends to $+\infty$.  
\end{proof}

\begin{theo}
\label{thm:nonLelongnumber}
Let $h$ be a singular hermitian metric on $-K_{X/Y}$ with semipositive curvature current. 
If the upper level set $P(h)$ of Lelong numbers is not dominant over $Y$, then $P(h)$ is empty 
$($that is, the Lelong number of $h$ is zero everywhere$)$.  
\end{theo}

\begin{proof}

We remark that $-K_{X/Y}$ is nef by $\mathbb{B}_-(-K_{X/Y}) \subseteq P(h)$ 
and Theorem \ref{thm:nefness}. 
We first show that $\mathcal{V}_m$ is hermitian flat. 
Note that, if $\mathcal{V}_m$ admits a positively curved singular hermitian metric $H$, 
the metric $H$ is automatically a smooth hermitian metric with flat Chern curvature 
by \cite{CP17} and \cite[Lemma 3.5]{HIM19}. 
To get a positively curved metric $H$ on $\mathcal{V}_m$, 
we take a singular hermitian metric $g$ on $\tilde{A}$ with semipositive curvature 
by Theorem \ref{r-viehweg}. 
Then, by Skoda's lemma and H\"order's inequality, we can easily check that 
$$
\I{h^{(1+m/p)}g^{1/p}}= \mathcal{O}_X 
\text { on } X \setminus \phi^{-1}(\phi (P(h))) \text{ for $p \gg 1$}
$$
since the singularities of $g$ can be killed by taking a sufficiently large $p$.
Note that the above property holds for any $m \in \mathbb{Z}_{>0}$  by the assumption. 
Hence, by \cite[Theorem 2.2]{CCM}, 
we can obtain the singular hermitian metric $H_{m, p}$ on $\mathcal{V}_m$ 
induced by the metric $h^{(p+m)}g$ on $-(p+m) K_{X/Y}+ \tilde{A}$ and the equality
$$
\mathcal{V}_m= \phi_*(pK_{X/Y} -(p+m) K_{X/Y}+ \tilde{A}). 
$$

For a contradiction, we assume that there exists a point $x_0 \in X$ with $\nu(h, x_0)>0$. 
We can  take a section $t$ of $(-mK_{X/Y}+\tilde{A})|_{X_y}$ whose valued at $x_0$ is non-zero 
since we may assume that it is globally generated on $X_y$ for any $m \geq 0$. 
By locally extending the section $t$, we can find a local section $s(y)$ of $\mathcal{V}_m$ 
defined on a (Euclidean) open neighborhood of $y_0:=\phi(x_0)$ such that $s(y_0)=t$. 
The norm $|s(y)|_{H_{m,p}}$ is a smooth function (in particular, it is bounded) 
since $H_{m,p}$ is automatically smooth.
On the other hand,  by the construction of $H_{m,p}$, 
we have 
$$
|s(y)|^2_{H_{m,p}}=\int_{X_y} |s(y)|^2_{B_p^{1-1/p} h^{(1+m/p)}g^{1/p}}, 
$$
where $B_p$ is the $p$-Bergman kernel on $pK_F + (-(p+m)K_F + \tilde{A})$. 
 The section $s(y_0)=t$ must be  zero at $x_0$  for a sufficiently large $m$,  
since the above norm is smooth and thus bounded. 
This is a contradiction to the choice of $t$. 
\end{proof}
\begin{thm}
\label{thm:semiampleness}
If $\mathbb B(-K_{X/Y})$ is not dominant over $Y$, then it is empty $($that is, $-K_{X/Y}$ is semi-ample$)$. 
\end{thm}

\begin{proof}
The stable base locus $\mathbb{B}(-K_{X/Y})$ coincides with 
the base locus of $-m_0K_{X/Y}$ for some $m_0 \in \mathbb{Z}_{>0}$. 
It is easy to see that $\nu(\varphi, x)>0$ if and only if $\varphi(x)=\infty$
in the case where a (quasi)-psh function $\varphi$ has analytic singularities. 
Hence the  singular hermitian metric $h$ on $-m_0K_{X/Y}$ defined by a basis of $H^0(X, -m_0K_{X/Y})$ satisfies 
that 
$$
\mathbb{B}(-K_{X/Y})={\rm{Bs}}(-m_0K_{X/Y})=P(h^{1/m}). 
$$
Then the assertion follows from Theorem \ref{thm:nonLelongnumber}. 
\end{proof}

\begin{rem}\label{rem-semi}
(1) In Theorem \ref{thm:semiampleness}, 
we do not mention of the direct images, 
but we show that $\phi_{*}(-mK_{X/Y})$ is \'etale trivializable 
when $-K_{X/Y}$ is semi-ample in 
Section \ref{subsection:semi-ample -K}. \\
(2) Theorem \ref{thm:2_intro} holds for klt pairs $(X, \Delta)$ 
by the same argument. 
\end{rem}

\section{Algebraic fiber spaces with semi-ample anti-canonical divisors} 
\label{subsection:semi-ample -K}
This section is devoted to the proof of Theorem \ref{thm:4_intro}. 

\begin{proof}[Proof of Theorem \ref{thm:4_intro}]
Let $\phi: X \to Y$ be a morphism with connected fibers between smooth projective varieties $X$ and $Y$ 
such that the relative anti-canonical divisor $-K_{X/Y}$ is semi-ample. 
For simplicity, we put $L:=-K_{X/Y}$.
The proof can be divided into six steps.  
\setcounter{step}{0}
\begin{step}[Explanation of our strategy]\label{step1}
 
By the main result of \cite{CH19, Cao19}, 
the fiber product of $X$ by the universal cover $\bar{Y}:=Y_{\rm{univ}} \to Y$ is isomorphic to the product $\bar{Y} \times F$. 
Furthermore, there exists a homomorphism $\rho: \pi_1(Y) \to \Aut{F}$ of groups so that 
$X$ can be constructed from the quotient $\bar{Y} \times F \cong \bar{X}$ 
by $\pi_{1}(Y)$ and $\rho$. 
Here $\pi_1(Y)$ naturally  acts on $\bar Y$ and also acts on $\bar Y \times F$ 
by $(y, p) \mapsto (\gamma \cdot y, \rho(\gamma) \cdot p) $, 
where $(y, p) \in \bar Y \times F$ and $\gamma \in \pi_1(Y)$. 
We will denote by $\bar{Y} / N$ (resp. $\bar{Y} \times F / N$) 
the quotient of $\bar{Y}$ (resp. $ \bar{Y} \times F$) 
corresponding to a subgroup $N \subseteq \pi_1(Y)$. 
Our situation can be summarized by the following commutative diagram and isomorphisms: 
\[\xymatrix{
\bar{Y} \times F \cong \bar{X}:=\bar{Y} \times_Y X\ar[d] \ar[r]\ar@{}[dr] & 
X' := \bar{Y} \times F / N\ar[r]\ar[d]
&
X \cong \bar{Y} \times F / \pi_1(Y) 
\ar[d]^{\phi} \\
\bar{Y}:=Y_{\rm{univ}} \ar[r] &Y' =\bar{Y} / N\ar[r]&Y\cong \bar{Y}/\pi_1(Y). \\
}\]

The following claim based on an elemental  observation gives our basic strategy.

\begin{claim}\label{claim1}
Let $N$ be a normal subgroup in $\pi_1(Y)$.  
Then the following conditions are equivalent:  
\begin{itemize}
\item[$\bullet$] The quotient group $\pi_1(Y)/N$ is a finite group and $N$ is contained in $\Ker{\rho}$. 

\item[$\bullet$] The quotient $X':=\bar{Y} \times F / N  \to X$ corresponding to  $N$ 
is a finite \'etale cover and $X'$ is isomorphic to the product $Y' \times F$, 
where $Y'$ is the quotient of $Y$ corresponding to $N$. 
\end{itemize}
\end{claim}
\begin{proof}
It is easy to check that $X'=\bar{Y} \times F / N  \to X$ is a finite morphism 
if and only if  $\pi_1(Y)/N$ is a finite group. 
In the case of $N \subseteq \Ker{\rho}$, 
the action $g:=\rho(\gamma): F \to F$ on $F$ induced  
by $\rho$ and an arbitrary loop $\gamma \in N$ in $Y$  
is trivial (that is, $g=\id_F$), 
and thus the corresponding quotient $\bar{Y} \times F / N  \to X$ is the product $Y/N \times F$. 
Conversely, when the quotient $\bar{Y} \times F / N$  is isomorphic to the product, 
an arbitrary loop $\gamma \in N$ trivially acts on $F$. 
\end{proof}

Our basic strategy is to show 
that $\pi_1(Y)/\Ker{\rho} \cong \Image{\rho} \subseteq \Aut{F}$ 
is a finite subgroup. 

\end{step}

\begin{step}[Comparison of semi-ample fibrations]\label{step2}

Let $H$ (resp. $G_y$) be the image of the semi-ample fibration associated to $mL$ 
(resp. the restriction $L|_{X_y}$ to a fiber $X_y$). 
For a fixed sufficiently large $m>0$, we have the morphisms with connected fibers 
$$
\Phi:=\Phi_{|mL|}:X \to H \quad \text{and} \quad \varphi:=\Phi_{|mL|_{X_y}|}:X_y \to G_y. 
$$
All the fibers $X_y$ are isomorphic to $F$ in our situation, 
but we will use the notations $X_y$ and $G_y$ (which looks like depending on $y$) 
to carefully treat the automorphisms of $F$.  
In this step, we will compare the above semi-ample fibrations and prove that 
the Stein factorization of the restriction of $\Phi$ to $X_y$ 
actually coincides with $\varphi$ 
by using the solution of Hacon-$\mathrm{M^{c}}$Kernan's question 
proved by Ejiri-Gongyo (see \cite{HM07, EG19} and see \cite{CCM} for its generalization). 

The main result of \cite{EG19} asserts that
\begin{align*}
H^{0}(X, mL) \to H^{0}(X_y, mL|_{X_y}) 
\end{align*}
is injective. 
Hence, for a basis $\{s_i \}_{i=1}^{k}$ of $H^{0}(X, mL) $, 
we can take sections $\{t_j \}_{j=1}^{\ell}$ of $H^{0}(X_y, mL|_{X_y}) $ 
such that $\{s_i|_{X_{y}}, t_j\}_{i=1, j_=1}^{k, \ell}$ is a basis $H^{0}(X_y, mL|_{X_y}) $. 
We may assume that $H$ (resp. $G_y$) is the image of 
the semi-ample fibration $\Phi$ (resp. $\varphi$) 
defined by $[s_1: \cdots: s_k]$ (resp. $[s_1: \cdots: s_k, t_1: \cdots: t_\ell]$). 
Now we consider the linear projection $p:\mathbb{P}^{k+\ell-1} \dashrightarrow \mathbb{P}^{k-1}$ to the first $k$  components 
and its restriction to $G_y \subseteq \mathbb{P}^{k+\ell-1}$. 
We remark that the linear projection is a rational map, 
but its restriction to $G_y \subseteq \mathbb{P}^{k+\ell-1}$ determines the morphism.  
We have the following commutative diagram: 
\[\xymatrix{
X\ar@{}[d]|{\bigcup}\ar[r]^-{\quad \Phi \quad} \ar@{}[dr] & H   \ar@{^{(}-_>}[r] &\mathbb{P}^{k-1} \\
X_y \ar[r] \ar[r]^-{\quad \varphi \quad}  & G_y \ar[u]_{p}  \ar@{^{(}-_>}[r]& \mathbb{P}^{k+\ell-1}.  \ar@{.>}[u]\\
}\]
Let $A$ (resp. $B$) be the ample divisor on $H$ (resp. $G_y$) 
defined by the restriction 
of the hyperplane bundle of $\mathbb{P}^{k-1}$ (resp. $\mathbb{P}^{k+\ell-1}$). 
By the definition, we have $B=p^* A$. 
On the other hand, we have 
$$
mL = \Phi^*A \quad \text{ and } \quad mL|_{X_y} = \varphi^*B 
$$
by the construction of $\Phi$ and $\varphi$.  
Hence, by considering the restriction of the first equality to $X_y$, 
we obtain 
$$
\varphi^*(p^* A) = mL|_{X_y} = \varphi^*B. 
$$
By the projection formula, we can obtain $p^*A=B$
since the morphism $\varphi$ is algebraic fiber space 
(that is, it is proper and its fibers are connected). 
It follows that $p$ must be a finite morphism 
since the pull-back of an ample divisor is ample again. 
This implies that $\varphi: X_y \to G_y$ and $G_y \to H$ 
gives the Stein factorization of $\Phi|_{X_y}: X_y \to H$.

If $p: G_y \to H$ is injective (that is, $p$ is isomorphic), 
the morphism defined by $p \mapsto (\phi(p), \Phi(p))$ 
determines the isomorphism $X \to Y \times H$ (which finishes the proof). 
However $p: G_y \to H$ is finite but not always one to one mapping (see Example \ref{ex:semi-ample}). 
For this reason, we have to construct an appropriate a finite \'etale cover of $Y$ (and also $X$) 
so that $p: G_y \to H$ is isomorphic. 
The key point here is the space of sections of $mL|_{X_y}$ is invariant, 
but the space of sections of $X$ is expanded, 
by taking the fiber product by \'etale covers of $Y$, 
which enables us to get $p: G_y \cong H$.  
\end{step}

\begin{step}[Relations between $\rho:\pi_1(Y) \to \Aut{F}$ and sections of $-K_{X/Y}$]\label{step3}
In this step, we reveal relations of $\rho:\pi_1(Y) \to \Aut{F}$ and sections of $-K_{X/Y}$ 
by applying the structure theorem of klt pairs with nef anticanonical bundles (see \cite[Theorem 1.3]{CCM}). 
This step is the core of Theorem \ref{thm:4_intro}.

Let $D$ be an effective divisor $D$ such that $D$ is $\mathbb{Q}$-linearly equivalent to $L$. 
Then, for a sufficiently small $\delta>0$, we can easily check that 
$(X, \delta D)$ is a klt pair and $-(K_{X/Y}+\delta D) \sim_{\mathbb{Q}} -(1-\delta)K_{X/Y}$ is nef. 
By applying  \cite[Theorem 1.3]{CCM} to the pair $(X, \delta D)$, 
we can conclude that  not only the manifold $X$ but also the pair 
$(\bar X, \delta \bar D)$ has the structure of the product, 
where $\bar D$ is the pull-back $\bar D := \pi_X^* D$ and $\pi_X: \bar X \to X$. 
Hence, for a section $s \in H^{0}(X, mL)$, we can find a section $s_F \in H^{0}(F, mL|_F)$ 
such that  
$$
\pi_X^* s = \pr_2^* s_F, 
$$
where $\pr_2: \bar X = Y \times F \to F$ is the second projection. 
Then it can be shown that the composite morphism $\Phi \circ \pi_X: \bar X \to H$, 
which is defined by sections $\{\pi_X^*s_i\}_{i=1}^k$, 
is actually determined by sections in $H^{0}(F, mL|_F)$. 
This implies that the restriction of $\Phi \circ \pi_X $ to the fiber $\{y\} \times F$ does not depend on $y \in \bar Y$. 
In other words, the composite morphism  $\Phi \circ \pi_X$ factors 
into the second projection $\pr_2: \bar X = Y \times F \to F$ and the composite morphism $\psi:=p \circ \varphi :F(\cong X_y) \to H$: 

\[
\xymatrix{
\bar X \ar[d]_{\pr_2} \ar[r]^-{\pi_X} \ar@{}[dr] & X   \ar[r]^-{\Phi} &H\\
F \ar[rru]^-{\psi } \ar[r] \ar[r]_-{\cong}  & X_y   \ar[r]_-{\varphi} &G_y.  \ar[u]_-{p}\\
}\]
By the above argument, for an arbitrary loop $\gamma$ in $Y$, 
the automorphism $g:=\rho(\gamma) \in \Aut{F}$ induced by $\rho: \pi_{1}(Y)  \to \Aut{F}$ 
always satisfies the following commutative diagram: 
\begin{equation}\label{comm1}
\vcenter{
\xymatrix{
F \ar[rr]^{g}  \ar[rd]^{\psi}  \ar[d]_{\varphi}&   & F \ar[ld]_{\psi} \ar[d]^{\varphi}\\
G_y  \ar[r]_{p}& H &  G_y\ar[l]^{p}. 
}
}
\end{equation}
\end{step}

\begin{step}[Proof of  Theorem \ref{thm:4_intro} (1)]\label{step4}
In this step, we will prove Theorem \ref{thm:4_intro} (1) under the assumption that $L|_F=-K_F$ is ample. 
It is sufficient for this purpose to prove that $\Image{\rho} \subseteq \Aut{F}$ is a finite subgroup by Claim \ref{claim1}.

We may assume that $\varphi: F \to G_y$ is isomorphic since $L|_F=-K_F$ is ample. 
Then $\psi=\varphi \circ p$ is finite since $p$ is so. 
Hence, by considering an unramified  point of $\psi$, 
we can find a (Euclidean) open subset $U_H \subseteq H$ 
such that  $\psi: V_{\lambda} \to U_H$ is isomorphic for any $\lambda \in \Lambda$, 
where $\psi^{-1}(U_H)=\coprod_{\lambda \in \Lambda} V_{\lambda}$ and 
$\Lambda$ is a finite set with $|\Lambda| = \deg \psi$. 
It follows that $g \in \Image{\rho}$ is an isomorphism over $H$ 
(that it, it preserves the morphism $\psi: F \to H$)
from the commutative diagram (\ref{comm1}). 
Hence we obtain the homomorphism
$$
\Image{\rho} \subseteq \Aut{F/H} \to \Aut{\psi^{-1}(U_H)/U_H}, 
$$
where $\Aut{F/H}$ (resp. $\Aut{\psi^{-1}(U_H)/U_H}$) is the set of automorphisms of $F$ (resp. $\psi^{-1}(U_H)$)
preserving $\psi: F \to H$ (resp. $\psi: \psi^{-1}(U_H) \to U_H$). 
By $\psi^{-1}(U_H)=\coprod_{\lambda \in \Lambda} V_{\lambda}$, 
we can show that the automorphisms $\Aut{\coprod_{\lambda \in \Lambda} V_{\lambda}/U_H}$ 
are just the permutations of $\Lambda$, 
and thus we have $|\Aut{\psi^{-1}(U_H)/U_H}| \leq |\Lambda| =(\deg p) !$. 
On the other hand, we can see that the above homomorphism is injective 
by an observation on the graph of $g \in \Image{\rho}$ and its irreducibility. 
Hence the conclusion (1) can be obtained by Claim \ref{claim1} and the finiteness of $\pi_1(Y)/\Ker{\rho}$. 
\end{step}

\begin{step}[Proof of \'etale trivializability of direct images]\label{step5}
In this step, we will prove that the direct image $\phi_*(-mK_{X/Y})$ 
is a trivial vector bundle 
after we take an appropriate finite \'etale cover $Y' \to Y$ and its base change. 
We go back to the commutative diagram (\ref{comm1}). 
The morphism $\varphi$ with connected fibers and the finite morphism $p$ 
give the Stein factorization not only for $\psi$ but also for $\psi \circ g $. 
By the uniqueness of the Stein factorization, 
the automorphism $g \in \Image{\rho}$ induces the automorphism of $G_y$ (which we denote by $\bar g$). 
This determines the homomorphism  
$$
\bar \rho: \pi_1(Y) \xrightarrow{\quad \rho \quad}  \Aut{F} \xrightarrow{\quad   \quad}  \Aut{G_y}. 
$$
By the same argument as in Step \ref{step4}, 
we can conclude that $ \pi_1(Y)/\Ker{ \bar \rho}\cong \Image{\bar \rho}$ is a finite group 
since $p$ is a finite morphism. 
Now we define $Y'$ (resp. $X'$) by the quotient of $\bar Y$ (resp. $\bar X$) corresponding to $\Ker{ \bar \rho}$. 
Then we have the following diagram: 
\[\xymatrix{
\bar{Y} \times F \cong \bar{X}\ar[d] \ar[r]\ar@{}[dr] & 
X' := \bar X / \Ker{\bar \rho}\ar[r] \ar[d]^{\phi'}
&X  \ar[d]^{\phi}  \ar[r]^{\Phi}& H\\
\bar{Y} \ar[r] &Y' =\bar{Y} / \Ker{\bar \rho}\ar[r]&Y. \\
}\]
By the construction, 
the morphism $Y' \to Y$ is finite and 
the induced morphism $\phi': X' \to Y'$ is the fiber product of $X$ by $Y' \to Y$ (see the proof of Claim \ref{claim1}). 
Hence, by replacing $\phi:X \to Y$ with $\phi': X' \to Y'$, 
we may assume that any loop in $Y$ trivially acts on $G_y$ (that is, $\Image {\bar \rho}=\{\id_{G_y}\}$).

By applying the result of \cite{EG19} to  our new morphism $\phi: X \to Y$ again (see Step \ref{step2}), 
we can see that the restriction map of sections of $-mK_{X/Y}$ to $X_y$ is injective. 
Hence it is sufficient to prove that any section in $H^{0}(X_y, -mK_{X_y})$ 
can be extended to a section  in $H^{0}(X, -mK_{X/Y})$. 
For a given section $t \in H^{0}(X_y, -mK_{X_y})$, 
we can obtain the extended section $\bar t=\pr_2^* t  \in H^{0}(\bar X, -mK_{\bar X/\bar Y}) $ 
by identifying $t$ with the section in $H^{0}(F, -mK_{F}) $. 
On the other hand, we can take $u \in H^{0}(G_y, B)$ such that 
$t=\varphi^*u$ by the definition of $\varphi$ and $B$. 
If a loop $\gamma$ in $Y$ is in $\Ker{\bar \rho}$, 
the induced automorphism $\bar g$ trivially acts on $G_y$, 
where  $g:=\rho(\gamma)$. 
Hence $\bar t=\pr_2^* t=\pr_2^* \varphi^*u$ is invariant under the actions of $\pi_1(Y)$. 
This implies that the extended section $\bar t=$ can be descended to the section in 
$H^{0}(X, -mK_{X/Y}) $. 
Therefore $\phi: X \to Y$ satisfies that $\phi_*(mL) = Y \times H^{0}(F, -mK_F)$.  

\end{step}

\begin{step}[Proof of  Theorem \ref{thm:4_intro} (2)]\label{step6}
In this step, we finally give a proof of Theorem \ref{thm:4_intro} (2). 
Our strategy is to reduce our situation to 
the case of $K_{X/Y}=\mathcal{O}_X$, 
in which the assertion holds by \cite[Theorem 5.8]{LPT18} (see\cite{Dru17}). 

By  Step \ref{step5}, we may assume that $\phi_*(mL) $ is a trivial vector bundle. 
In this case, since the morphism $p: G_y \to H$ is isomorphic, 
we identify $\varphi: X_y \to G_y$ with the restriction of $\Phi: X \to H$ to the fiber $X_y$. 
For the fiber $P_h:=\Phi^{-1}(h)$ at a point $h \in H$, 
the intersection $P_h \cap X_y$ coincides with the fiber of $\varphi: X_y \to G_y$ at $h \in H$. 
In particular, for a general point $h \in H$, 
the fiber $P_h$ is a smooth projective variety and 
the restriction $\phi|_{P_h}: P_h \to Y$ is also a surjective morphism with connected fibers. 
Note that the irregularity of a general fiber $\phi|_{P_h}: P_h \to Y$ is equal to zero 
by the assumption of Theorem \ref{thm:4_intro} (2). 
Further its relative canonical bundle $mK_{P_h/Y}$ is trivial. 
Indeed, since $P_h$ is a fiber of the semi-ample fibration associated to $mL$, 
we have $mL|_{P_h}=\mathcal{O}_{P_h}$ and $K_X |_{P_h}=K_{P_h}$, 
and thus we obtain $mK_{P_h/Y}=\mathcal{O}_{P_h}$. 
Hence, for a general point $h \in H$, the algebraic fiber space $P_h \to Y$ has the structure of the product 
by taking  the base change by a finite \'etale cover $Y' \to Y$ 
since the restriction $\phi|_{P_h}: P_h \to Y$ satisfies the assumptions of 
\cite[Theorem 5.8]{LPT18}. 

On the other hand, for an automorphism $g=\rho(\gamma) \in \Image{\rho}$, 
the induced automorphism $\bar g \in \Image{\rho}$ trivially acts on $H$ in our case, 
and thus $g $ preserves the fiber $P_h \cap F$. 
This implies that the restriction  $g$ to $P_h \cap F$ gives 
the automorphism $g_h:= g|_{P_h \cap F} \in \Aut{P_h \cap F}$. 
This determines the homomorphism 
$$
\rho_h: \pi_{1}(Y)  \xrightarrow{\quad \rho \quad} \Image{\rho}  \xrightarrow{\quad  \quad}  \Aut{P_h\cap F} 
\text{ defined by } \gamma \mapsto g=\rho(\gamma) \mapsto g_h. 
$$
Our goal is to prove that $\Image{\rho}$ is a finite group (see Claim \ref{claim1}). 
By the above argument with \cite[Theorem 5.8]{LPT18} and Claim \ref{claim1}, 
we have already shown that $\Image{\rho_h}$ is a finite subgroup in $\Aut{P_h\cap F} $
for a general point $h$. 
Therefore it is sufficient for our goal to show that 
$\Image{\rho}  \to  \Aut{P_h\cap F}$ is injective for some $h \in H$. 
This is proved by the following claim: 

\begin{claim}\label{claim2}
For an automorphism $g \in \Image{\rho}$ induced by a loop in $Y$ and $\rho$, 
the subset $A_g \subseteq H$ defined by 
$$
A_g:=\{ h \in H \,|\, g_h = \id_{P_h \cap F}\}. 
$$
is a Zariski-closed subset in $H$. 
Further, if $g \not = \id_{F}$, then $A_g$ is 
a proper  Zariski-closed subset in $H$. 
\end{claim}
\begin{proof}
It is easy to check the latter conclusion. 
Indeed, in the case of $A_g=H$, 
we have $g_h=g|_{P_h \cap F} $ for any $h \in H$, 
and thus $g$ trivially acts on $F$ (that is, $g = \id_{F}$). 

For a fixed $g  \in \Image{\rho}$, 
we define the (not necessarily irreducible)  subvariety $\mathscr{B}$ by 
$$
\mathscr{B}:=\{x \in F \,|\, g(x)=x\} \subseteq F. 
$$
Let $m$ be the relative dimension of $\varphi: F \to H$. 
Now we consider the morphism $\varphi|_{\mathscr{B}}: \mathscr{B} \to H$. 
We remark that the condition of $h \in A_g$ is equivalent to 
the condition of $\dim (\mathscr{B}_h )=m$,  
where $\mathscr{B}_h$ is the fiber at $h\in H$. 
On the other hand, we have 
$$
m \geq \dim (\mathscr{B} \cap F_h) = \dim (\mathscr{B} \cap P_h \cap F) 
\text{ \quad and \quad } 
\mathscr{B}_h =\mathscr{B} \cap P_h \cap F
$$
for a general $h \in H$. 
Further it can be shown that 
$$
\{ h \in H \, | \,  \dim( \mathscr{B}_h) \geq m\}
$$
is a Zariski-closed set 
by the upper semi-continuity of fiber dimensions. 
This completes the proof. 
\end{proof}

The above claim  finishes the proof of  Theorem \ref{thm:4_intro} (2). 
Indeed, let $A_0$ be a (proper) Zariski-closed set in $H$ such that the fiber $X_h$ at $h \not \in A_0$ is a projective manifold. 
Note that $\pi_1(Y)$ is an (at most) countable set 
since any loop can be approximated by polygonal lines. 
Then, for a very general point $h$ with $h \not \in A_0 \cup \cup_{\id\not= g \in \Image{\rho}} A_g$, 
the map $\Image{\rho}  \to  \Aut{P_h\cap F}$ is injective by the definition of $A_g$. 
Therefore $\Image{\rho_h}$ is a finite subgroup, and thus so is $\Image{\rho}$ . 
\end{step}
\end{proof}
\section{Examples} \label{section:ex}
In this section, we construct examples 
of algebraic fiber spaces having nef (semipositive, or semi-ample) relative anti-canonical divisors, 
which helps us to understand this paper.  

\begin{ex}[Algebraic fiber spaces with semi-ample $-K_{X/Y}$] \label{ex:semi-ample}
This example tells us that Theorem \ref{thm:4_intro} does not hold 
without taking \'etale covers. 

Let $Y$ be a smooth projective curve of genus at least one, 
and let $D$ be a divisor on $Y$ such that $\deg D\ge0$ and $D\not\sim0$. 
Set $\mathcal E:=\mathcal O_Y \oplus \mathcal O_Y(-D)$ 
and $X:=\mathbb P(\mathcal E)$. 
Let $\phi:X=\mathbb P(\mathcal E)\to Y$ denote the natural projection. 
Let $C_0 \subseteq X$ be the section corresponding to the quotient 
$\mathcal E \twoheadrightarrow \mathcal O_Y(-D)$. 
Then $ -K_{X/Y} \sim 2C_0 +\phi^*D$ (cf. \cite[V, Lemma~2.10]{Har77}), 
and thus it follows from the projection formula that, for each $m\ge0$, 
\begin{align}
\phi_*\mathcal O_X(-mK_{X/Y})
& \cong \phi_*\mathcal O_X(2mC_0)  \otimes \mathcal O_Y(mD)
\cong S^{2m}(\mathcal E)\otimes \mathcal O_Y(mD) \notag
 \\ & \cong \left(\bigoplus_{i=0}^{2m}\mathcal O_Y(-iD) \right) 
\otimes \mathcal O_Y(mD)
\cong \bigoplus_{i=0}^{2m}\mathcal O_Y((m-i)D). 
\label{cong:1}
\end{align}

If $D$ is a torsion element in $\mathrm{Cl}(X)=\mathrm{Div}(X)/\sim_{\mathrm{lin}}$ of order $n$,  
then we have 
\begin{itemize}
\item[(i)] 
$-K_{X/Y}$ is nef, since $\phi^*D$ is nef and $ C_0^{\cdot 2} =-\deg D =0 $, 
\item[(i\hspace{-1pt}i)] 
$2nC_0\in |-nK_{X/Y}|$, and 
\item[(i\hspace{-1pt}i\hspace{-1pt}i)] 
there is $E \in |-nK_{X/Y}|$ with $E \ne nC_0$, since $\dim|-nK_{X/Y}|=2$ by (\ref{cong:1}). 
\end{itemize}
This implies that $|-nK_{X/Y}|$ is free, and thus $-K_{X/Y}$ is semi-ample.
However $\phi:X=\mathbb P(\mathcal E)\to Y$ itself is not a product 
and the direct images $\phi_*-mK_{X/Y}$ are not trivial. 

We can directly check the assertion of Theorem \ref{thm:4_intro} for this example. 
Indeed, the $n$-torsion divisor $D$ defines an \'etale cyclic cover $\pi:Z\to Y$ such that $\pi^*D \sim 0$. 
Hence $\pi^*\mathcal E \cong \mathcal O_Z \oplus \mathcal O_Z$, 
and $X\times_Y Z$ is isomorphic to $\mathbb P^1\times_{\mathbb C}Z$. 
\end{ex}
\begin{ex}[Algebraic fiber spaces with semipositive but not semi-ample $-K_{X/Y}$]
\label{ex:nef but not semi-ample}
Let the notation be as in Example~\ref{ex:semi-ample}. 
If $D$ is a divisor such that $\deg D=0$ and that $D$ is not a torsion element in $\mathrm{Cl}(X)$,  
then we see that $-K_{X/Y}$ is semipositive 
(that is, it admits a smooth hermitian metric with semipositive curvature),  
but it is not semi-ample by 
$$
0 \overset{\textup{by (\ref{cong:1})}}{=} \kappa(-K_{X/Y}) 
< \mathrm{nd}(-K_{X/Y}) =1.
$$ 
Then it is easy to see that the direct images 
$$
\phi_{*}(-mK_{X/Y}) \cong \bigoplus_{i=0}^{2m}\mathcal O_Y((m-i)D)
$$
are hermitian flat by $c_1(\mathcal O_Y(-D))=0$. 

\end{ex}
\begin{ex}[Lc pairs $(X, D)$ with semi-ample and big \textup{$-(K_{X/Y}+D)$}] 
\label{ex:nef big but not ample}
This example says that 
Theorem \ref{thm:4_intro} does not hold for an lc pair $(X, D)$ 
even if $-(K_{X/Y}+D)$ is semi-ample and big. 
It is known that $\mathbb{B}_{+}(-(K_{X/Y}+D))=X$ holds if $(X, D)$ is klt and if $-(K_{X/Y}+D)$ is nef, 
but this example also says that it is not true for lc pairs. 
However the statement of Corollary~\ref{cor:dominate} 
(that is, $\mathbb{B}_{+}(-(K_{X/Y}+D))$ is dominant over $Y$) 
is still true for lc pairs.

Let the notation be as in Example~\ref{ex:semi-ample}. 
Let $C_1$ be the section of $\phi$ 
corresponding to the quotient $\mathcal E\twoheadrightarrow\mathcal O_Y$. 
Then $C_1 \sim C_0 +\phi^*D$ and $C_0\cdot C_1 =0$ 
(cf. \cite[V, Proposition~2.9]{Har77}). 
We consider the lc pair $(X,C_1)$. Put $L:=K_{X/Y}+C_1$. 
Then we have 
$$
-L \sim 2C_0+\phi^*D -C_1 \sim C_0,$$ 
and thus $-L$ is nef.
Assume that $\deg D>0$. 
Then $C_0^{\cdot 2}>0$, so $-L$ is big. 
We have $\mathbb B_+(-L)=C_1$ thanks to \cite[Theorem~C]{ELMNP2}. 
Further $-L$ is semi-ample by  
$$
\mathbb B(-L)\subseteq C_0 \cap \mathbb B_+(-L) =C_0 \cap C_1 = \emptyset.
$$  
Note that there is no finite cover $\pi:Z\to Y$ of $Y$ 
such that $X\times_Y Z$ is isomorphic to $\mathbb P^1 \times_{\mathbb C}Z$ as $Z$-schemes. 

\end{ex}

\begin{ex}\label{kodaira}
This example tells us that Theorem \ref{thm:4_intro} does not hold for K\"ahler manifolds. 

Let $(z, w)$ be the standard coordinate of $\mathbb{C}^2$. 
We consider the compact complex torus $X:=\mathbb{C}^2/\Gamma$ 
defined by the lattice generated by 
$$
(0, 1), (0,\tau), (\tau, 0), (1, \alpha). 
$$
Here $\alpha:=a+\tau b$, 
$\tau$ is a complex number whose imaginary part is non-zero, and $a, b$ are real numbers with $0 \leq a, b \leq 1$. 
Note that $X$ is always K\"ahler, but not projective for general $a, b$.  
Then the natural first projection  $X \to Y$ to 
the elliptic curve $Y:=\mathbb{C}/\langle 1, \tau\rangle$ defined by $1$ and $\tau$ 
is  a locally trivial morphism 
with the fiber that is isomorphic to $\mathbb{C}/\langle 1, \tau\rangle$. 
Further its relative canonical bundle is trivial. 
In particular,  the morphism $X \to Y$ satisfies the assumptions of \cite[Theorem 5.8]{LPT18} (see \cite{Dru17}). 
However $X$ is projective if $X$ is a product up to finite \'etale covers. 
This is a contradiction to the case of $a, b$ being general. 

Theorem \ref{thm:4_intro} was proved 
for morphisms with trivial relative canonical bundle in \cite[Lemma6.4]{Dru17}, 
in which the key point is the existence of fine moduli of polarized abelian varieties with level structures. 
The lack of fine modulus causes this example. 
\end{ex}

As in \cite{HM07, CH19, EG19, CCM}, 
it is natural to consider an application of our argument
to maximally rationally connected (MRC) fibrations, 
but we do not deal with it, 
because this paper only treat morphisms. 
Recall that, by definition, the MRC fibrations associated to a variety
are almost holomorphic maps, 
which may not be represented by a morphism. 
In fact, the authors leaned from Kento Fujita 
that there exists a normal projective variety whose MRC fibrations 
cannot be represented by a morphism (Example~\ref{ex:Fujita's example}). 
The authors do not know such varieties except this example. 
Fujita constructed a normal projective 3-fold $X'$ with Picard number one 
that is uniruled but \textit{not} rationally connected. 
\begin{ex} \label{ex:Fujita's example}
Let $S$ be a smooth projective surface such that 
\begin{itemize}
\item $\rho(S)=1$ and 
\item the MRC fibration associated to $S$ is represented by $\mathrm{id}:S\to S$. 
\end{itemize}
For example, a very general Abelian surface or hypersurface of degree $d\ge 4$ in $\mathbb P^3$ 
satisfies the above conditions. 
Take a very ample divisor $\mathcal L$ on $S$ and 
a smooth curve $C$ in $|\mathcal L|$. 
We define the notation by the following commutative diagram
whose all squares denote base changes:
\begin{align}
\xymatrix{ C_X \ar[r] \ar[d] \ar@/_30pt/[dd]_-\sigma^-\cong & B_X \ar[d] & \\
F_X \ar[d] \ar[r] & X \ar[d]^-{p_1} \ar[r]^-{p_2} & \mathbb P^1 \ar[d] \\
C \ar[r] & S \ar[r] & \mathrm{Spec}\,\mathbb C
}
\end{align}
Here, $B_X$ is a general member in the complete linear system 
of $\mathcal M_k:=p_1^*\mathcal L^k \otimes p_2^*\mathcal O(1)$ for $k\ge 2$. 
We have $\mathcal N_{C_X/X}\cong \mathcal M_k|_{C_X} \oplus \sigma^* \mathcal L$. 
 
Let $\pi:Y\to X$ be the blow-up along $C_X$ and 
consider the elemental transformation
\begin{align}
\xymatrix{ & Y \ar[dr]^-\tau \ar[dl]_-\pi & \\
	X \ar[dr]_-{p_1} & & Z \ar[dl]^-p \\
	  & S. &
}
\end{align}
Let $E$ denote the exceptional divisor of $\pi$. 
Set $B:=\pi^{-1}_*B_X$ and $F:=\pi^{-1}_*F_X$.  Since 
\begin{align}
B -F \sim \pi^*(B_X -F_X) \in \pi^*|\mathcal M_{k-1}| 
\end{align}
and $B\cap F=\emptyset$, we see that $|B|$ is free, 
and thus it defines a morphism $\varphi:Y\to X'$ with $\varphi_*\mathcal O_Y \cong \mathcal O_{X'}$, 
which contracts $F$ to a point. 

We show that $\varphi$ induces an isomorphism from 
$Y\setminus F$ to $X'\setminus\varphi(F)$. 
Suppose that there is a curve $\gamma\subseteq Y$ such that $\gamma\not\subseteq F$ and $\varphi(\gamma)$ is a point. 
Then we have 
$$
0\ge-\gamma\cdot F=\gamma\cdot (B-F)=\gamma\cdot\pi^*\mathcal M_{k-1}
=\pi(\gamma)\cdot\mathcal M_{k-1}. 
$$
Then it follows that $\pi(\gamma)$ is one point since $\mathcal M_{k-1}$ is ample. 
Therefore $\gamma\subseteq E\cong \mathbb P(\mathcal N_{C_X/X}^*)$, 
and the choice of $\gamma$ implies that it is contained in  
the section $F\cap E \cong \mathbb P(\sigma^*\mathcal L^*)\subseteq \mathbb P(\mathcal N_{C_X/X}^*)$, 
so $\gamma \subseteq F$, which is a contradiction.

Hence $\varphi:Y\to X'$ contracts every fiber of $\tau:Y\to Z$ to a point, 
and thus $\varphi$ factors through $Z$. 
The induced morphism $\psi:Z\to X'$ is not an isomorphism, 
since $\tau(F)$ is not a point. 
Then we have 
\begin{align}
	\rho(X')\le\rho(Z)-1 =\rho(X)-1 = 2-1 =1, 
\end{align}
so we obtain $\rho(X')=1$. 
The MRC fibration associated to $X'$ is the composite of 
$X' \dashrightarrow X\xrightarrow{p_1} S.$
\end{ex}

\renewcommand{\thesection}{\Alph{section}} \setcounter{section}{0}
\section{Appendix} \label{section:appendix}

\subsection{Proof of the statements in subsection \ref{subsection:semistability}} \label{subsection:analytic proof1}

In this subsection, 
we give a proof for Proposition \ref{r-c10}, 
Theorem \ref{r-viehweg}, and \ref{r-thm:locallytrivial}. 


\begin{proof}[Proof of Proposition \ref{r-c10}]
We consider the direct image of the divisor 
\begin{align*}
pK_{X/Y} + qN \Delta +L=(mN+p)K_{X/Y}
\overbrace{-mN(K_{X/Y} +\Delta )}
^{\textup{with  $g_{\varepsilon'}^{m}$}} +
\overbrace{(m+q)N \Delta }^{\textup{with   $h_{\Delta}^{(m+q)N} $}} +
\overbrace{L}^{\textup{with  $h_L^{1-\varepsilon}H_L^{\varepsilon}$}}
\end{align*}
equipped with singular metrics defined as follows: 
The metric $g_{\varepsilon}$, which is obtained from the assumption of the restricted base locus, 
is a singular metric on $-N(K_{X/Y}+\Delta)$ such that 
$\sqrt{-1}\Theta_{ g_{\varepsilon'} } \ge - \varepsilon' \omega_{X}$ and 
$g_{\varepsilon}|_{X_y}$ is smooth on $X_y$ for a general point $y \in Y$. 
The metric $h_{\Delta}$ is the canonical singular metric on $\Delta$. 
The metric $H_L$, which is obtained from the $\phi$-bigness of $L$, 
is a singular metric on $L$ so that 
$\sqrt{-1}\Theta_{H_L} + k \phi^{*}\omega_Y  \geq \delta \omega_X
$
for some $k$ and $\delta>0$. 
Then $h := g_{\varepsilon}^{m}h_{\Delta}^{(m+q)N} h_{L}^{1-\varepsilon}H_{L}^{\varepsilon}$ 
satisfies that 
\begin{align*}
\sqrt{-1}\Theta_{h}
=&-\varepsilon'm \omega_X + (1-\varepsilon)\phi^{*}\theta+
 \varepsilon (\delta \omega_X-k\phi^{*}\omega_Y)\\
\geq&  (1-\varepsilon)\phi^{*}\theta - \varepsilon k \phi^{*}\omega_Y
\end{align*}
for $\varepsilon \gg \varepsilon'>0$. 
On the other hand, for a sufficiently large $m \in \mathbb{Z}_{>0}$, 
we can see that 
$ \I{h^{1/(mN+p)}|_{y}} = \mathcal{O}_{X_y}$ for a  general $y \in Y$,  
since  we have $\I{h_{L}^{1/mN+q}}=\I{H_{L}^{1/mN+q}}=\mathcal{O}_{X}$, 
$g_{\varepsilon'}$ is smooth on $X_y$, and
$\bigl(X, \frac{(m+q)N}{mN+p}\Delta \bigr)$ is a klt pair.
The induced metric on the direct image sheaf 
satisfies the desired conclusion  
(see \cite[Lemma 5.25]{CP17}, \cite{PT18, BP08}.)
\end{proof}

\begin{proof}[Proof of Theorem \ref{r-thm:locallytrivial}] 
By \cite[Proposition 2.7, Theorem 2.8]{CCM} and Proposition \ref{nonvert}, 
it is sufficient to show that 
$\phi_{*}(p\tilde{A}) $ and $\phi_{*}(qN\Delta + p\tilde{A}) $ 
are numerically flat vector bundles 
for any $p \in \mathbb{Z}_{>0}$ and any $q \in \mathbb{Z}_{\ge 0}$. 
The proof of the numerical flatness 
is the same as in \cite{Cao19, CH19, CCM}, 
but we give the proof here for reader's convenience.

Let $A$ be a sufficiently ample divisor on $X$. 
Set $r := \rank(c_1(\phi_{*}A))$ and
$$
\tilde{A} := A - \frac{1}{r}\phi^{*}(\det \phi_{*}(A)).
$$
Note that we may assume that 
$\phi_{*}(A) $ is locally free by Proposition \ref{flatness} 
and $\tilde{A}$ is a Cartier divisor by \cite[Prop 3.9]{CH19}.

Now we prove that 
$$
\mathcal{V}_{m,q,p}:= \phi_{*}(-mN(K_{X/Y}+\Delta )+qN\Delta+ p\tilde{A})
$$ 
is numerically flat 
for any $m,q \in  \mathbb{Z}_{\ge0}$ and any $p \in \mathbb{Z}_{>0}$, 
by applying Theorem \ref{r-viehweg} which is proved later. 
It follows that $\mathcal{V}_{m,q,p}$ is weakly positively curved 
(and thus $c_1( \mathcal{V}_{m,q,p})$ is pseudo-effective) 
from Theorem \ref{r-viehweg} and Proposition \ref{r-c10}. 
By Theorem \ref{r-viehweg}, there exists a $\phi$-exceptional effective divisor $E$ 
such that 
$$
-mN(K_{X/Y}+\Delta )+qN\Delta+p\tilde{A}+ E - \frac{1}{r_{m,q,p}}\phi^{*}(\det \mathcal{V}_{m,q,p})
$$
is pseudo-effective. By applying Proposition \ref{r-c10} 
to $\theta$ representing the first Chern class $c_{1}(\mathcal{V}_{m,q,p})$ and 
the divisor 
$$
p\tilde{A}+E
 =  mN K_{X/Y}+(m-q)N\Delta + (-mN(K_{X/Y}+\Delta )+qN\Delta+p\tilde{A}+ E),
$$
we obtain $c_1(\phi_{*}( p\tilde{A} )) \ge \frac{r_{0,0,p}}{r_{m,q,p}}c_1(\mathcal{V}_{m,q,p})$.
This implies that $c_1(\mathcal{V}_{m,q,p}) = 0$
from  $0 = c_1(\phi_{*}( \tilde{A} )) \ge \frac{1}{pr_{0,0,p}}c_1(\mathcal{V}_{0,0,p})$ by Theorem \ref{r-viehweg} and Proposition \ref{r-c10}. 




\end{proof}


\begin{proof}[Proof of Theorem \ref{r-viehweg}]
The basic idea is the same as in \cite[Proposition 3.15]{Cao19}. 
Note that $\phi$ is smooth in codimension 1 by Proposition \ref{flatness}.
Hence there exists a Zariski-closed set $Z \subseteq X$ such that $\codim Z \ge 2$
and $\phi|_{X \setminus Z}: X \setminus Z \rightarrow Y \setminus \phi(Z)$ is smooth.


Let $X^{(r)}$ be a desingularization of 
the $r$-times fiber product $ X \times_{Y}X\times_{Y} \cdots \times_{Y} X $.
Let $\pr_i : X^{(r)} \rightarrow X$ be the $i$-th projection and 
$\phi^{(r)}: X^{(r)} \rightarrow Y$ be the induced morphism.
Set $V := X^{(r)} \setminus  (\cap \pr_{i}^{-1}(X \setminus Z))$. 

By the same argument as in \cite[Proposition 3.15]{Cao19}, 
there exist effective divisors $D_1$, $D_2$ supported in $V$ such that
$$
K_{X^{(r)}/Y} =\sum_{i=1}^{r} \pr_{i}^{*}K_{X/Y} + D_{1}-D_{2}.
$$
From the assumption on the restricted base locus,  we obtain
$$
\phi^{(r)}\Bigl(\mathbb{B}_{-} \bigl(  \sum_{i=1}^{r} -\pr_{i}^{*} (K_{X/Y} + \Delta)  \bigr)\Bigr)\subseteq 
\cup_{i=1}^{r} \phi^{(r)}\Bigl(\mathbb{B}_{-}\bigl( - \pr_{i}^{*} (K_{X/Y} +\Delta)\bigr ) \Bigr)\neq Y.
$$

For the divisor $L:= \sum_{i=1}^{r}\pr_{i}^{*}L$, 
we have the canonical map
$$
\det(\phi_{*}(L)) \rightarrow \otimes^{r}  \phi_{*}(L) \cong \phi^{(r)}_{*}(L_{r}) \text{ on } Y_L, 
$$
where $Y_L$ be the maximum Zariski-open set such that $\phi_{*}(L)$ is  locally free. 
Hence we have the nonzero section of $\phi^{(r)}_{*}(L_{r}) \otimes \det(\phi_{*}(L))^{ \vee}$ on $Y_L$, 
and thus, by taking double dual, we have
$$
0 \neq  H^{0}(Y_L, \phi^{(r)}_{*}(L_{r})^{\vee\vee} \otimes \det(\phi_{*}(L))^{ \vee}) = H^{0}(Y, \phi^{(r)}_{*}(L_{r})^{\vee\vee} \otimes \det(\phi_{*}(L))^{ \vee}).
$$
By  \cite[III. Lemma 5.10]{Nak}, 
there exists a $\phi^{(r)}$-exceptional effective divisor $E_r$
such that 
\begin{equation}
\label{locfree}
\phi^{(r)}_{*}(L_{r})^{\vee\vee} = \phi^{(r)}_{*}(L_{r} + E_r).
\end{equation}
Thus $ L' := L_{r} + E_r -  { \phi^{(r)}}^{*}\det(\phi_{*}(L))$
is an effective divisor of $X^{(r)}$. 
Since $ L' $ is also $\phi^{(r)}$-big, we have: 
\begin{itemize}
\item $ L' $ has a singular hermitian metric $h_{L'}$ with semipositive curvature current.
\item $ L' $ has a singular hermitian metric $H_{L'}$ such that $\sqrt{-1}\Theta_{L',H_{L'}} \ge \omega_{X^{(r)}} - \phi^{*}\omega_{Y}$ 
for some K\"ahler form $\omega_{X^{(r)}}$  (resp. $ \omega_{Y}$) on $X^{(r)}$ (resp. $Y$). 
\end{itemize}
 We take a sufficiently large $m \in \mathbb{Z}_{>0}$ so that 
 $\mathcal{J}(h_{L'}^{\frac{1}{m}} | _{X^{(r)} _y}) = 
 \mathcal{J}(H_{L'}^{\frac{1}{m}} | _{X^{(r)} _y}) = \mathcal{O}_{X^{(r)} _y}$ for a general point $y \in Y$.
Then we have: 
\begin{claim}
\label{caov}
There exists an ample divisor $A_Y$ of $Y$ such that 
$$
H^{0}(X^{(r)},  mpD_{1}+pL'+{ \phi^{(r)}}^{*}A_Y ) \rightarrow H^{0}(X^{(r)}_{y},  mpD_{1}+pL'+{ \phi^{(r)}}^{*}A_Y |_{X^{(r)}_{y}}) 
$$
is surjective 
for any $p \in \mathbb{Z}_{>0}$ and a general point $y \in Y$. 
\end{claim}

\begin{proof}

We first take an ample divisor $A$ of $Y$ so that 
 $A$ admits a smooth hermitian metric $ h_{A}$ such that  $\sqrt{-1}\Theta_{h_{A}} - \omega_{Y} >0$. 
Set $A_Y := 2A$. 
For any $p \in \mathbb{Z}_{>0}$, 
we have
\begin{align*}
& mpD_{1}+pL'+{ \phi^{(r)}}^{*}A_Y =  mpK_{X^{(r)}/Y}  + \\ 
& \Bigl(\underbrace{- mp\sum_{i=1}^{r}\pr_{i}^{*}(K_{X/Y} +\Delta) }_{\text{with }h_{\varepsilon}}
+mp \underbrace{\sum_{i=1}^{r} \pr_{i}^{*}\Delta}_{\text{with }H_{\Delta}}
 +  mp\underbrace{D_{2}}_{\text{with }h_{D_2}} + \underbrace{pL'}_{\text{with }h_{L'}^{p-1}H_{L'}}+
 \underbrace{{ \phi^{(r)}}^{*}A  }_{\text{with } {\phi^{(r)}}^{*}h_{A}} \Bigr) +{ \phi^{(r)}}^{*}A. 
\end{align*}
By \cite[Theorem 2.10]{Cao19},
it is sufficient to show that there exists 
a singular hermitian metric $h$ on the first term in the right hand side 
with semipositive curvature current such that 
$\mathcal{J}(h^{\frac{1}{mp}} | _{X^{(r)} _y}) = \mathcal{O}_{X^{(r)} _y}$ for a general point $y \in Y$.
We define singular hermitian metrics $h_{\varepsilon}$, $H_{\Delta}$, and $h_{D_2}$
by the following way: 
\begin{itemize}
\item $-mp\sum_{i=1}^{r} \pr_{i}^{*}(K_{X/Y} +\Delta) $ admits a singular hermitian metric $h_{\varepsilon}$ such that 
$\sqrt{-1}\Theta_{  h_{\varepsilon} } \ge - \varepsilon \omega_{X^{(r)}}$ and $h_{\varepsilon}|_{X^{(r)}_y}$ is smooth on $X^{(r)}_y$ for a general point $y\in Y$ 
by the assumption on the restricted base locus. 

\item $\sum_{i=1}^{r} \pr_{i}^{*}\Delta$  admits the singular hermitian metric  
$H_{\Delta} := \prod_{i=1}^{r} \pr_{i}^{*}h_{\Delta}$, 
where the canonical singular hermitian metric $h_{\Delta}$ on the effective divisor $\Delta$. 
Then we have $\mathcal{J}(H_{\Delta}|_{X^{(r)} _y}) = \mathcal{O}_{X^{(r)} _y}$
for a general point $y\in Y$ by the klt condition. 

\item $D_2$ admits the canonical singular hermitian metric $h_{D_2}$. 
Then it follows that $h_{D_2}|_{X^{(r)}_y}$ is smooth on $X^{(r)}_y$ for a general point $y\in Y$.
since $D_{2}$ is supported in $V$. 
\end{itemize}
Then $h= h_{1} H_{\Delta}^{mp}h_{D_2}^{mp} h_{L'}^{p-1}H_{L'}({\phi^{(r)}}^{*}h_{A})$ 
satisfies the desired condition, that is, 
$$
\sqrt{-1}\Theta_{  h } \ge  - \omega_{X^{(r)}}  + \omega_{X^{(r)}} - { \phi^{(r)}}^{*}\omega_{Y}+ \sqrt{-1}\Theta_{  { \phi^{(r)}}^{*}h_{A}} \ge 0$$
and $\mathcal{J}(h^{\frac{1}{mp}} | _{X^{(r)} _y}) = \mathcal{O}_{X^{(r)} _y}$
for a general point $y \in Y$. 
This completes the proof.
\end{proof}

We may regard $ X\setminus Z$ as a subset of $X^{(r)}$ 
with  the diagonal map $j : X\setminus Z \rightarrow X^{(r)}$.
We have $H^{0}(X_y,  pL'+{ \phi^{(r)}}^{*}A_Y |_{X_y}) \neq0$ 
for a general point $y \in Y$ and sufficiently large $p \in \mathbb{Z}_{>0}$, 
and thus we can conclude that $H^{0}(X \setminus Z,  pL'+{ \phi^{(r)}}^{*}A_Y |_{X \setminus Z})\neq 0$
by Claim \ref{caov} and the following diagram:
\begin{equation*}
\xymatrix@C=40pt@R=30pt{
H^{0}(X^{(r)},  mpD_{1}+pL'+{ \phi^{(r)}}^{*}A_Y ) \ar@{->>}[r]
\ar[d]^{j^{*}}& H^{0}(X^{(r)}_{y},  pL'+{ \phi^{(r)}}^{*}A_Y |_{X^{(r)}_{y}})  \ar@{->>}[d]
 \\  
 H^{0}(X \setminus Z,  pL'+{ \phi^{(r)}}^{*}A_Y) |_{X \setminus Z}) 
 \ar[r]^{} &  H^{0}(X_y,  pL'+{ \phi^{(r)}}^{*}A_Y |_{X_y})& 
 }
\end{equation*}
On the other hand, from $\codim Z \ge 2$, 
it follows that 
$$0 \not=
H^{0}(X \setminus Z,  pL'+{ \phi^{(r)}}^{*}A_Y |_{X \setminus Z})=  H^{0}(X,  prL+pE -p \phi^{*}\det(\phi_{*}(L))+ \phi^{*}A_Y ). 
$$
Here $E$ is the $\phi$-exceptional effective divisor of $X$ such that 
$E|_{X \setminus Z}  = E_r |_{X \setminus Z} $. 
We can see that $rL+E - \phi^{*}\det(\phi_{*}(L))$ is pseudo-effective by taking $p \rightarrow +\infty$.

If $\phi_{*}(L)$ is locally free, $E$ can be taken as a zero divisor from $E_r = 0$ in Equation (\ref{locfree}). 
\end{proof}

\subsection{Analytic proof for Corollary~\ref{cor:dominate}} \label{subsection:analytic proof2}

In this subsection, we give an analytic proof for Corollary~\ref{cor:dominate}. 
Let $\phi : X \rightarrow Y$ be a surjective morphism with connected fiber 
between projective manifolds over $\mathbb{C}$. 
Let $\Delta$ be an effective $\mathbb{Q}$-divisor on $X$.

\begin{theo}
If $(X,\Delta)$ is an lc pair, 
then $\mathbb B_+(-(K_{X/Y}+\Delta))$ and $\mathbb B_-(-(K_{X/Y}+\Delta+\phi^*L))$
are dominant over $Y$ for any ample divisor $L$ on $Y$ 
unless $Y$ is one point. 
\end{theo}

\begin{proof}

Let $\omega_{X}$ be a K\"ahler form of $X$ and 
$\omega_{Y}$ be a K\"ahler form of $Y$ such that $\omega_{X} \ge \phi ^{*}\omega_{Y}$.

We assume that $\mathbb{B}_{+}\left( -(K_{X/Y} +\Delta) \right)$ is not dominant over $Y$. 
Take $N \in \mathbb{Z}_{>0}$ such that $N \Delta$ is Cartier and 
an ample divisor $A$ of $X$ such that $N\Delta + A$ is ample.
Fix $m \in \mathbb{Z}_{>0}$ such that $\mathbb{B}_{+}( - (K_{X/Y} + \Delta)) = \mathbb{B}\left( -mN(K_{X/Y} +\Delta)-A \right) $.
Then we consider 
$$
\mathcal{O}_{Y}=\phi _{*}(mNK_{X/Y} +
\overbrace{(-mN(K_{X/Y} +\Delta)-A)}^{\text{with $h_1$}} +
\overbrace{(N\Delta +A)}^{\text{with $h_2$}} +\overbrace{mN(1-1/m)\Delta)}^{\text{with $h_3^{mN}$}} 
$$
and singular hermitian metrics defined as follows:   
\begin{itemize}
\item $h_1$ is a singular hermitian metric on $-mN(K_{X/Y} + \Delta) -A $ 
such that $h_1|_{X_y}$ is smooth for a general point $y \in Y$.
\item $h_2$ is a smooth hermitian metric on the ample divisor $N\Delta+A$ 
with  $\sqrt{-1}\Theta_{h_2} \ge \varepsilon \omega_{X}$ for some $\varepsilon >0$.
\item $h_3$ is the canonical singular metric on $(1-1/m)\Delta$. 
\end{itemize}
Then $h := h_{1}h_{2}h_{3}^{mN}$ satisfies that 
$\sqrt{-1}\Theta_{h} \ge  \varepsilon \omega_{X} \ge \varepsilon \phi^{*}\omega_{Y} $
and $\mathcal{J}(h^{\frac{1}{mN}} |_{X_y}) = \mathcal{O}_{X_y}$ for a general point $y \in Y$.
By \cite[Theorem 2.2(1)]{CCM}, 
the induced metric $H$ on $\mathcal{O}_{Y}$ satisfies that 
$\sqrt{-1}\Theta_{H} \ge \varepsilon \omega_{Y} $, 
which is a contradiction.

Now we assume that $\mathbb B_-(-(K_{X/Y}+\Delta+\phi^*L))  $ is not dominant over $Y$.
Let $h_{L}$ be a smooth hermitian metric on $L$ with positive curvature.
By the same way as in the above argument, 
we can see that the direct image of  
$$ 2A=mNK_{X/Y} 
 +(-mN(K_{X/Y} +\Delta + \phi^*L) +A ) +mN\phi^*L+(N\Delta +A) +mN(1-1/m)\Delta$$
admits a singular hermitian metric $H_{m}$ such that 
$\sqrt{-1}\Theta_{H_m} \succeq mN\sqrt{-1}\Theta_{h_{L}} \otimes \id$. 
This implies that $\det(\phi_{*}(2A)) - mNL$ is pseudo-effective for any $m \in \mathbb{Z}_{>0}$, which is a contradiction.
\end{proof}
\bibliographystyle{alpha}

\end{document}